\documentclass{amsart}% option [reqno] puts equat numb to the right

\usepackage[mathscr]{eucal}% So-called Euler fonts, allows \mathscr
\usepackage{amssymb}% allows special arrows like >-> e.g.
\usepackage[usenames,dvipsnames]{color}
\usepackage[normalem]{ulem}% allows \sout
\usepackage{amsthm}% allows theorem* , e.g.
\usepackage{bbold}% allows \unit \mathbb{1}
\usepackage{enumerate}% allows easy change of label
\usepackage{array}% allows array
\usepackage{amsmath}% allows pmatrix

\numberwithin{equation}{section}% makes equat numb contain the section
\usepackage[all]{xy}
\SelectTips{cm}{}
\newdir{ >}{{}*!/-10pt/\dir{>}}

\usepackage[colorlinks=true,linkcolor={Brown},citecolor={Brown}]{hyperref}

\swapnumbers %pour que le num�ro apparaisse devant le th�or�me

\newtheorem{Thm}[equation]{Theorem}
\newtheorem{Prop}[equation]{Proposition}
\newtheorem*{Prop*}{Proposition}
\newtheorem{Lem}[equation]{Lemma}
\newtheorem{Cor}[equation]{Corollary}
\theoremstyle{remark}
\newtheorem{Def}[equation]{Definition}
\newtheorem{Not}[equation]{Notation}
\newtheorem{Exa}[equation]{Example}
\newtheorem{Cons}[equation]{Construction}
\newtheorem{Conv}[equation]{Convention}
\newtheorem{Hyp}[equation]{Hypothesis}

\newtheorem{Rem}[equation]{Remark}
\newtheorem{Que}[equation]{Question}
\newcommand{\nc}{\newcommand}
\nc{\dmo}{\DeclareMathOperator}

\dmo{\Ab}{Ab}
\dmo{\Ann}{Ann}% annihilator
\dmo{\coev}{coev}
\dmo{\coker}{Coker}
\dmo{\Der}{D}% ground notation for derived categories
\dmo{\End}{End}
\dmo{\ev}{ev}
\dmo{\Hom}{Hom}
\dmo{\id}{id}
\dmo{\Id}{Id}
\dmo{\img}{Im}
\dmo{\Ind}{Ind}
\dmo{\Ker}{Ker}
\dmo{\Loc}{Loc}
\dmo{\modname}{mod}%
\dmo{\Mod}{Mod}% sheaves of modules
\dmo{\opname}{op}
\dmo{\Qcoh}{Qcoh}% quasi-coherent modules
\dmo{\Res}{Res}
\dmo{\rmH}{H}
\dmo{\SH}{SH}% ground name for cat of spectra
\dmo{\smallperf}{perf}% ground exponent for ``perfect''
\dmo{\Spc}{Spc}
\dmo{\stab}{stab}% stable category of fin. gen. mod.
\dmo{\Stab}{Stab}% stable category of non-fin. gen. mod.
\dmo{\supp}{supp}
\nc{\adjto}{\rightleftarrows}
\nc{\adj}{\dashv}
\nc{\aka}{{a.\,k.\,a.}\ }
\nc{\bbE}{\mathbb{E}}
\nc{\bbM}{\mathbb{M}}
\nc{\bbQ}{\mathbb{Q}}
\nc{\bbZ}{\mathbb{Z}}
\nc{\calI}{\mathcal{I}}
\nc{\calO}{\mathcal{O}}
\nc{\cat}[1]{\mathscr{#1}}%or: \nc{\cat}[1]{\mathcal{#1}}
\nc{\colim}{\mathop{\mathrm{colim}}}
\nc{\Dperf}{\Der^{\smallperf}}% derived category of perfect compl
\nc{\eg}{{\sl e.g.}}
\nc{\Endcat}[1]{\End_{\cat #1}}
\nc{\gm}{\mathfrak{m}}% prime m
\nc{\Homcat}[1]{\Hom_{\cat #1}}
\nc{\hook}{\hookrightarrow}
\nc{\ideal}[1]{\langle #1\rangle}
\nc{\ie}{{\sl i.e.}\ }
\nc{\ihomcat}[1]{\ihom_{\cat #1}}
\nc{\ihom}{{\mathsf{hom}}} %{{\underline{\hom}}}
\nc{\into}{\mathop{\rightarrowtail}}
\nc{\inv}{^{-1}}
\nc{\isotoo}{\overset{\sim}{\,\too\,}}
\nc{\isoto}{\overset{\sim}{\,\to\,}}
\nc{\kk}{k}%{\Bbbk}
\nc{\matrice}[1]{\begin{pmatrix} #1 \end{pmatrix}}
\nc{\Mid}{\,\big|\,}
\nc{\mmod}{\modname\kern-0.1em\text{-}}%
\nc{\MMod}{\Mod\kern-0.1em\text{-}}%
\nc{\onto}{\mathop{\twoheadrightarrow}}
\nc{\op}{^{\opname}}
\nc{\otoo}[1]{\overset{#1}{\,\too\,}}
\nc{\oto}[1]{\overset{#1}\to}
\nc{\ourfrac}[2]{\genfrac{}{}{0pt}{}{\scriptstyle #1}{\scriptstyle #2}}
\nc{\Paul}[1]{{\color{OliveGreen}[#1]}}
\nc{\potimes}[1]{^{\otimes #1}}% tensor power
\nc{\qquadtext}[1]{\qquad\textrm{#1}\qquad}
\nc{\quadtext}[1]{\quad\textrm{#1}\quad}
\nc{\restr}[1]{_{|_{\scriptstyle #1}}}% to restrict a map
\nc{\SET}[2]{\big\{\,#1\Mid#2\,\big\}}
\nc{\sstab}{\,\text{--}\stabname}%
\nc{\too}{\mathop{\longrightarrow}\limits}
\nc{\unit}{\mathbb{1}}% unit for \otimes
\nc{\Greg}[1]{{\color{CarnationPink}#1}}% from latin 'carnatio'..., very 'Reservoir Dog'
\nc{\Henning}[1]{{\color{Blue}#1}}

\nc{\Gout}[1]{\Greg{\sout{#1}}}%not the best abbreviation -- :))

\dmo{\im}{Im}
\dmo{\fpname}{fp}
\dmo{\Nil}{Nil}
\dmo{\yoneda}{h}
\dmo{\Add}{Add}
\dmo{\Tor}{Tor}

\dmo{\smodname}{\stab}%
\dmo{\sMod}{\Stab}%
\dmo{\smod}{\stab}%

\nc{\MT}{\MMod\cat{T}^c}
\nc{\mT}{\mmod\cat{T}^c}
\nc{\MF}{\MMod\cat{F}^c}
\nc{\mF}{\mmod\cat{F}^c}
\nc{\MtT}{\MMod\widetilde{\cat{T}}^c}
\nc{\mtT}{\mmod\widetilde{\cat{T}}^c}
\nc{\mcA}{\cat{A}}% the macro ``\cat'' for categories (we can change that one if you want)
\nc{\mcB}{\cat{B}}
\nc{\topspace}{\vphantom{I^{I^I}}}
\nc{\bottomspace}{\vphantom{j_{j_j}}}
\nc{\fp}{^{\fpname}}
\nc{\boneda}{\bar\yoneda}
\nc{\bbe}{\mathbb{e}}
\nc{\bbf}{\mathbb{f}}
\nc{\bat}[1]{\bar{\cat{#1}}}
\nc{\hcat}[1]{\hat{\cat{#1}}}
\nc{\sfK}{\mathsf{K}}
\nc{\rf}{(\mathbf{rf1})}
\nc{\rfc}{(\mathbf{rf1^c})}
\nc{\rff}{(\mathbf{rf2})}
\nc{\tff}{$\otimes$-faithful}
\nc{\SpcT}{\Spc(\cat{T}^c)}% most used
\nc{\SpcF}{\Spc(\cat{F}^c)}% most used

\begin{document}

%------------------------------------------------------------------------------

\title{Tensor-triangular fields: Ruminations}
\author{Paul Balmer}
\author{Henning Krause}
\author{Greg Stevenson}

\address{Paul Balmer, Mathematics Department, UCLA, Los Angeles, CA 90095-1555, USA}
\email{balmer@math.ucla.edu}
\urladdr{http://www.math.ucla.edu/~balmer}

\address{Henning Krause, Universit\"at Bielefeld, Fakult\"at f\"ur Mathematik,
Postfach 10\,01\,31, 33501 Bielefeld, Germany}
\email{hkrause@math.uni-bielefeld.de}
\urladdr{http://www.math.uni-bielefeld.de/~hkrause/}

\address{Greg Stevenson, School of Mathematics and Statistics,
University of Glasgow,
University Place,
Glasgow G12 8QQ
}
\email{gregory.stevenson@glasgow.ac.uk}
\urladdr{http://www.maths.gla.ac.uk/~gstevenson/}

\begin{abstract}
We examine the concept of field in tensor-triangular geometry. We gather examples and discuss possible approaches, while highlighting open problems. As the construction of residue tt-fields remains elusive, we instead produce suitable homological tensor-functors to Grothendieck categories.
\end{abstract}

%\date{2017 August 21}

\thanks{P.\,Balmer supported by Humboldt Research Award and NSF grant~DMS-1600032.}

\maketitle

%------------------------------------------------------------------------------
%------------------------------------------------------------------------------

\vskip-\baselineskip\vskip-\baselineskip
\tableofcontents

%------------------------------------------------------------------------------

\section*{Foreword}

%------------------------------------------------------------------------------

This article is partly speculative. We investigate the concept of `fields' beyond algebra, in the general theory of tensor-triangulated categories. Readers can refresh their marvel at the beauty and ubiquity of tensor-triangulated categories by consulting standard references like~\cite{HoveyPalmieriStrickland97}. We further assume some minimal familiarity with the topic of \emph{tensor-triangular geometry}; see~\cite{BalmerICM} or~\cite{Stevenson16pp}.

A guiding principle of the tt-geometric philosophy is \emph{horizontal diversity}: We are not after a sophisticated notion of field tailored for homotopy theory, or for algebraic geometry, or for representation theory. We seek a broad notion of tt-field that must not only work in \emph{all} these examples but also in $KK$-theory of C${}^*$-algebras, in motivic theories, and in all future examples yet to be discovered.

Stable $\infty$-categorists and Quillen-model theorists should feel at home here, as long as they enjoy the occasional tensor product. We focus on homotopy categories and their tensor for the tt-axioms travel lightly between different mathematical realms. It is however possible that some of the problems we present below may one day be solved via richer structures.

\begin{Hyp}
\label{hyp:big}%
Throughout the paper a \emph{`big' tt-category} refers to a rigidly-compactly generated tensor-triangulated category~$\cat{T}$, see~\cite{BalmerFavi11,BalmerKrauseStevenson17pp}. By assumption, $\cat{T}$ admits all coproducts, its compact objects and its rigid objects coincide, and the essentially small subcategory~$\cat{T}^c$ of compact-rigids generates~$\cat{T}$ as a localizing subcategory. (Rigid objects are called `strongly dualizable' in~\cite{HoveyPalmieriStrickland97}.)
\end{Hyp}

The article is organized as follows. Section~\ref{se:intro} begins with a discussion of the concept of tt-field and the related question of producing a \emph{tt-residue field} for local tt-categories. That section ends with a review of our main results. In Sections~\ref{se:MT}, \ref{se:E} and~\ref{se:max} we propose an approach via abelian categories, replacing the evanescent triangular residue functor by a \emph{homological} one. We return to tt-fields in the final Sections~\ref{se:field} and~\ref{se:fairyland}. The former is dedicated to proving properties of tt-fields. In the latter, we assume the existence of a tt-residue field for a local tt-category~$\cat{T}$ and show how it matches the abelian-categorical results of Sections~\ref{se:MT}-\ref{se:max}. Appendix~\ref{app:Grothendieck} contains purely Grothendieck-categorical results, for reference.

%------------------------------------------------------------------------------
\goodbreak
\section{Introduction to tt-fields and statement of results}
\label{se:intro}%
\medbreak
%------------------------------------------------------------------------------

Let us start by recalling an elementary pattern of commutative algebra, namely the reduction of problems to local rings, and then to residue fields; schematically:
\[
\xymatrix@C=5em{
\{\textrm{\ global data\ }\}
 \ar[r]^-{\textrm{localization}}
& \{\textrm{\ local data\ }\}
 \ar[r]^-{\textrm{quotient}}
 \ar@/^1.5em/@{..>}[l]^-{\textrm{descent}}
& \{\textrm{\ residue field data\ }\}\,.
 \ar@/^1.5em/@{..>}[l]^-{\textrm{Nakayama}}
}
\]
The method for recovering global information from local data is generically called descent. On the other hand, a useful way to recover local information from the residue field is Nakayama's Lemma, at least under suitable finiteness conditions. We would like to assemble a similar toolkit in general tt-geometry.

Localization of tensor-triangulated categories is well understood (see Remark~\ref{rem:local}) and descent has been extensively studied. The present article focusses on the right-hand part of the above picture. Given a local tt-category~$\cat{T}$ (Definition~\ref{def:local}) we want to find a tt-category~$\cat{F}$ together with a tt-functor
\[
F\colon\cat{T}\to \cat{F}
\]
such that $\cat{F}$ is a `tt-field'. Clarifying the latter notion is one first difficulty. Such a tt-functor $F$ should satisfy some form of Nakayama, a property which most probably means that $F$ is conservative (detects isomorphisms) on compact-rigids. This conservativity on~$\cat{T}^c$ matches the behavior in examples and implies that $\SET{x\in\cat{T}^c}{F(x)=0}$ is the unique closed point~$(0)$ in the spectrum~$\Spc(\cat{T}^c)$.

So what should `tt-fields' be? In colloquial terms, a tt-field~$\cat{F}$ should be `an end to the tt-road': It should not admit non-trivial localizations, nor non-trivial `quotients'. In particular, we should not be able to `mod out' any non-zero object nor any non-zero morphism by applying a tt-functor going out of~$\cat{F}$.

Although vague, this preliminary intuition is sufficient to convince ourselves that some well-known tt-categories should be recognized as tt-fields. Of course, in commutative algebra, the derived category of a good old commutative field~$\kk$ should be a tt-field. Also, the module category over Morava $K$-theory (in a structured enough sense) should be a tt-field in topology. In these examples, the homotopy groups are `graded fields' $\kk[t,t\inv]$ with $k$ a field and $t$ in even degree, and all modules are sums of suspensions of the trivial module. (See some warnings in Remark~\ref{rem:warn}.) Consideration of so-called `cyclic shifted subgroups', or `$\pi$-points', in modular representation theory of finite groups lead us to another class of candidates:
\begin{Prop*}
Let $p$ be a prime, $G=C_p$ the cyclic group of order~$p$ and~$\kk$ a field of characteristic~$p$. Let $\cat{F}=\Stab(\kk G)$ be the stable module category, \ie the additive quotient of the category of $\kk G$-modules by the subcategory of projective modules.

Then every non-zero coproduct-preserving tt-functor $F\colon\cat{F}\to \cat{S}$ is faithful.
\end{Prop*}
This is Proposition~\ref{prop:C_p}, where further `field-like' properties of~$\Stab(\kk C_p)$ are isolated. This result tells us that one should probably accept $\cat{F}=\Stab(\kk C_p)$ as a tt-field although it is quite different from `classical' fields. Granted, for $p=2$, the above category $\cat{F}\cong \MMod\kk$ is very close to a `good old field'. However, for $p>2$, the homotopy groups in~$\cat{F}$ form the Tate cohomology ring $\pi_{-*}(\unit)=\Homcat{F}(\unit, \Sigma^{*}\unit)\cong\hat\rmH{}^{*}(C_p,\kk)$, which is the graded ring $\kk[t^{\pm1},s]/s^2$ with $t$ in degree~2 and $s$ in degree~1. In particular, $s$ is a nilpotent element in~$\pi_*$ which cannot be killed off by any non-trivial tt-functor out of~$\cat{F}$, because of the above proposition.

In other words, one should renounce some traditional definitions of fields. Topologists sometimes call fields those (nice enough\footnote{\,ring spectra, highly structured, $E_\infty$, or else. This is not the point debated here.}) rings over which every module is a sum of suspensions of the ring itself. This property still holds in the example of~$\cat{F}=\Stab(\kk C_p)$ for $p=2$ or~$p=3$. However, for every $p\ge 5$ there are objects in~$\cat{F}$ which are not direct sums of suspensions of the $\otimes$-unit~$\unit$. A generalization of the old topological definition has been used in the motivic setting in recent work of Heller and Ormsby~\cite{HellerOrmsby16pp}; they call `field' a (nice enough) ring whose modules are sums of $\otimes$-invertibles. This generalized definition is motivated by the existence of additional `spheres' in the motivic setting, namely $\mathbb{G}_{\textrm{m}}$ or~$\mathbb{P}^1$. However, for $p\ge 5$ our $\cat{F}=\Stab(\kk C_p)$ is stubbornly not a field in the Heller-Ormsby sense either, for its only $\otimes$-invertibles are the `usual' spheres~$\Sigma^*\unit$. See Proposition~\ref{prop:C_p}.

In summary, because of simple examples in modular representation theory, we cannot define tt-fields as those tt-categories whose homotopy groups are graded fields, nor as those tt-categories in which every object is a sum of spheres, or even a sum of $\otimes$-invertibles. We need something more flexible.

Dwelling on the instructive example of~$\cat{F}=\Stab(\kk C_p)$ for a moment longer, we see that every object of~$\cat{F}$ is a coproduct of finite-dimensional indecomposable objects whose dimension is invertible in~$\kk$. See Proposition~\ref{prop:C_p} again. It follows that $\cat{F}$ is a tt-field in the sense of the following tentative  generalization:
\begin{Def}
\label{def:field}%
A non-trivial `big' tt-category~$\cat{F}$ will be called a \emph{tt-field} if every object $X$ of~$\cat{F}$ is a coproduct $X\simeq \coprod_{i\in I}x_i$ of compact-rigid objects $x_i\in\cat{F}^c$ and if every non-zero $X$ is \emph{\tff}, meaning that the functor
\[
X\otimes-\colon\cat{F}\to \cat{F}
\]
is faithful (if $X\otimes f=0$ for some morphism~$f$ in~$\cat{F}$ then $f=0$).
\end{Def}

\begin{Rem}
A similar definition was suggested in~\cite[\S\,4.3]{BalmerICM} by only asking that every non-zero object of~$\cat{F}^c$ be \tff. Under the assumption that every object of~$\cat{F}$ is a coproduct of compacts, the two versions are evidently the same. We do not know any example of~$\cat{F}$ in which every non-zero (compact) object is \tff\ without having the other property. So it is possible that Definition~\ref{def:field} contains some redundancy.
\end{Rem}

We establish basic properties of such tt-fields in Section~\ref{se:field}. For instance, we show in Theorem~\ref{thm:field-ihom} that being a field is equivalent to the internal hom functor $\ihomcat{F}(-,X)$ being faithful for every non-zero object~$X$ -- a very simple formulation.

A tt-field~$\cat{F}$ must have a minimal spectrum: $\Spc(\cat{F}^c)=\{*\}$ (Proposition~\ref{prop:spcF}). In other words, every non-zero object generates the whole category. This matches the intuition that a field should be very small. It also explains how, for~$\cat{T}$ local, the expected `Nakayama property' of the residue tt-functor $F\colon\cat{T}\to \cat{F}$, namely the conservativity of~$F\colon\cat{T}^c\to\cat{F}^c$, simply means that the induced map $\Spc(F)\colon\Spc(\cat{F}^c)\to \Spc(\cat{T}^c)$ sends the unique point of~$\Spc(\cat{F}^c)$ to the closed point of~$\Spc(\cat{T}^c)$.

We are thus led to the quest for tt-residue fields.
\begin{Que}
\label{que:residue}%
Given a local tt-category~$\cat{T}$, is there a coproduct-preserving tt-functor $F\colon\cat T\to \cat{F}$ to a tt-field~$\cat{F}$ (Definition~\ref{def:field}), with~$F$ conservative on~$\cat{T}^c$\,?
\end{Que}

In this generality, Question~\ref{que:residue} remains an open problem, and most probably a difficult one. However, it is an important problem in a number of respects. For the purist this question represents a unification of our understanding of the `big' tt-world, and is part of a circle of questions concerning tt-rings and passage to closed subsets in tt-geometry. For the pragmatist this would yield powerful tools, e.g.\ nilpotence theorems and classifications, that could be used in specific examples.

When we said above that Question~\ref{que:residue} remains an open problem we meant it; it was not a rhetorical device building up to the announcement of a solution. Rather, we propose here a palliative approach via abelian categories which works unconditionally. We hope both the purist and pragmatist find this intermediary appealing. On one hand it seems to hint at new vistas in tt-geometry and works uniformly. On the other it is a sufficient framework for proving an extremely general tensor-nilpotence theorem \cite{Balmer17pp} which can be used for computations.

Before giving an overview of our approach and the contents of this article, let us add a warning and a further comment.

\begin{Rem}
\label{rem:warn}%
In topology, the Morava $K$-theories at the prime~2 do not admit a homotopy-commutative ring structure. So it is not even clear that their modules form a \emph{tensor}-triangulated category. In the same vein, in representation theory, each point of the support variety of a finite group~$G$ is detected by a `$\pi$-point', which comes with an exact functor from the stable category of~$G$ to a tt-field~$\Stab(\kk C_p)$. Again, this restriction is not always a \emph{tensor}-functor, unless one tinkers with the tensor in~$\Stab(\kk G)$. Both examples point to the possibility that one might need to adjust the role of the tensor product in the construction of tt-residue fields.
\end{Rem}

As mentioned in the foreword, it is possible that the theory of $\infty$-categories, or that of model categories, could help us solve Question~\ref{que:residue}. Mathew shows in~\cite{Mathew17a} how to produce (topological) residue fields when working over a field of characteristic zero and when~$\cat{T}$ is a stable $\infty$-category with only even homotopy groups. Even this very special case appears remarkably difficult.

\smallbreak
\begin{center}
*\ *\ *
\end{center}

Let us now say a word about the announced approach to residue fields via abelian categories.
Recall that there exists a \emph{restricted-Yoneda} functor
\begin{equation}
\label{eq:res-yoneda}%
\vcenter{
\xymatrix@C=2em@R=0em{
\yoneda\colon & \cat{T} \kern1.2em \ar[rr]^-{\displaystyle}
&& \kern2em \cat{A}:=\MT \kern.8em
\\
& X \kern1em \ar@{|->}[rr]
&& \kern1.1em \hat X := \Homcat{T}(-,X)\restr{\cat{T}^c} \kern-2em
}}
\end{equation}
from the tt-category~$\cat{T}$ to the Grothendieck category $\cat{A}=\MT$ of $\cat{T}^c$-modules, \ie contravariant additive functors from~$\cat{T}^c$ to abelian groups. The functor~$\yoneda$ preserves coproducts, is conservative and is \emph{homological}, meaning that it maps exact triangles to exact sequences. Moreover, the category~$\cat{A}$ admits a colimit-preserving tensor which makes~$\yoneda\colon\cat{T}\to \cat{A}$ into a tensor functor. See~\cite[App.\,A]{BalmerKrauseStevenson17pp} for details.

We attack the problems presented above from the angle of the module category~$\cat{A}$ when $\cat{T}$ is local. There are two related facets to this idea. First, we want to produce a `residue abelian category' $\bat{A}$ together with a coproduct-preserving homological tensor-functor~$\boneda\colon\cat{T}\to \bat{A}$ which is conservative on~$\cat{T}^c$ and in such a way that~$\bat{A}$ is `very small'. Second, in case there miraculously exists a tt-residue field $F\colon\cat{T}\to \cat{F}$ at the triangular level, we would like to relate the corresponding categories of modules, $\MT$ and~$\MF$, and the~$\bat{A}$ constructed above.

Our first series of results establishes the unconditional existence of such an~$\bat{A}$.
\begin{Thm}
\label{thm:summary}%
Let $\cat{T}$ be a `big' tt-category as in Hypothesis~\ref{hyp:big}, which we assume \emph{local} (Definition~\ref{def:local}). Then there exists a (possibly non-unique) functor
\[
\boneda\colon\cat{T}\to\bat{A}
\]
to an abelian tensor category~$\bat{A}$, satisfying all of the following properties:
\begin{enumerate}[\rm(a)]
\smallbreak
\item
The category~$\bat{A}$ is a Grothendieck category and is locally coherent (Remark~\ref{rem:loc-coh}). The tensor~$\otimes$ on~$\bat{A}$ commutes with colimits in each variable and restricts to a tensor on the subcategory~$\bat{A}\fp$ of finitely presented objects of~$\bat{A}$. In particular, its unit~$\bar\unit\in\bat{A}\fp$ is finitely presented.
\smallbreak
\item
The functor~$\boneda$ is coproduct-preserving, homological and strict monoidal. It is conservative on~$\cat{T}^c$: if $f\colon x\to y$ in~$\cat{T}^c$ is such that $\boneda(f)$ is an isomorphism then $f$ is an isomorphism; equivalently, if $\boneda(x)=0$ for $x\in\cat{T}^c$ then $x=0$.
\smallbreak
\item
The category~$\bat{A}$ is `very small' in the following sense: Every non-zero finitely presented object of~$\bat{A}$ (\eg\ $\boneda(x)$ for $x\in\cat{T}^c$) generates~$\bat{A}\fp$ as a Serre $\otimes$-ideal and generates~$\bat{A}$ as a localizing $\otimes$-ideal. The endomorphism ring of the $\otimes$-unit $\End_{\bat{A}}(\bar\unit)$ is local of Krull dimension zero, \ie its maximal ideal is a nilideal.
\smallbreak
\item The image under~$\boneda$ of every $X\in\cat{T}$ is flat
  in~$\bat{A}$. The image of every $x\in\cat{T}^c$ is
  finitely presented and rigid. Every injective object of~$\bat{A}$ is
  the image under~$\boneda$ of a pure-injective object of~$\cat{T}$,
  and in particular is flat.
%
%\smallbreak
%\item
%If the localizing $\otimes$-ideal of $\cat{T}$ supported at the closed point of $\Spc~(\cat{T}^c)$ is minimal then there is a unique pair $(\bat{A},\boneda)$ with the above properties.
\end{enumerate}
\end{Thm}

\begin{proof}
  This result will occupy most of
  Sections~\ref{se:MT}-\ref{se:max}. The category $\bat{A}$ is
  constructed as the Gabriel quotient of the module
  category~$\cat{A}=\MT$ by a localizing subcategory that is
  generated by a maximal Serre $\otimes$-ideal
  subcategory~$\cat{B}\subset\cat{A}\fp$ of finitely presented
  objects, which meets $\yoneda(\cat{T}^c)$ trivially, see
  Proposition~\ref{prop:max}. The quotient is recalled in
  Proposition~\ref{prop:A/B}, from which part~(a) follows. Then,
  specifically: Part~(b) is Proposition~\ref{prop:all-generate}\,(a);
  Part~(c) is Corollary~\ref{cor:all-generate} and
  Theorem~\ref{thm:end-nil}; Part~(d) is essentially Corollary~\ref{cor:A/B} which proves the flatness statements and that injectives in $\bat{A}$ come from $\cat{T}$. It remains to show that objects of $\cat{T}^c$ are sent to finitely presented rigid objects. The quotient functor is monoidal by Corollary~\ref{cor:A/B} which gives rigidity and preservation of being finitely presented is Proposition~\ref{prop:A/B}.
%  Part~(e) is Corollary~\ref{cor:unique}.
\end{proof}

We produce the desired $\bat{A}$ by cooking up a suitable ideal in $\cat{A}$ to kill. We moreover show, in Corollary~\ref{cor:unique}, that if the localizing $\otimes$-ideal of $\cat{T}$ supported at the closed point of $\Spc(\cat{T}^c)$ is minimal, and $\cat{T}$ satisfies some further technical conditions, then the ideal we construct is \emph{unique}. This suggests that one might attempt to use the collection of such ideals as a refinement of $\Spc(\cat{T}^c)$ in order to understand localizations of $\cat{T}$.

Our second collection of results can be summarized as follows: if $\cat{T}$ admits a tt-residue field then it produces, in a natural way, a `residue abelian category' as above and this abelian category is very close to the category of modules over the tt-field.

\begin{Thm}
\label{thm:residue}%
Let $\cat{T}$ be a local `big' tt-category and $F\colon\cat{T}\to\cat{F}$ be a coproduct-preserving tt-functor into a tt-field~$\cat{F}$ (Definition~\ref{def:field}). Suppose that $F$ is conservative on~$\cat{T}^c$ and surjective up to direct summands. (See Hypothesis~\ref{hyp:res}.) Then there exists a ring-object~$\bbE$ in~$\cat{T}$ satisfying all the following properties:
\begin{enumerate}[\rm(a)]
\smallbreak
\item
The object~$\bbE$ is pure-injective, hence $\hat{\bbE}=\yoneda(\bbE)$ is injective in~$\cat{A}$.
\smallbreak
\item
The idempotent-completion of the Kleisli category of free~$\bbE$-modules in~$\cat{T}$ is equivalent to~$\cat{F}$.
\smallbreak
\item
The category of $\hat{\bbE}$-modules in~$\cat{A}=\MT$ is equivalent to the category $\MF$.
\smallbreak
\item
The Serre $\otimes$-ideal~$\cat{A}\fp\cap \Ker(\hat{\bbE}\otimes-)$ of finitely presented objects annihilated by~$\hat{\bbE}$ is maximal among those Serre $\otimes$-ideals in~$\cat{A}\fp$ which meet $\yoneda(\cat{T}^c)$ trivially. The quotient $\bat{A}=\cat{A}/\Ker(\hat{E}\otimes-)$ and the functor~$\boneda\colon\cat{T}\to\cat{A}\onto \bat{A}$ satisfy all the properties of $\bat{A}$ and~$\boneda$ listed in Theorem~\ref{thm:summary}.
\end{enumerate}
\end{Thm}

\begin{proof}
Again, this is a summary of Section~\ref{se:fairyland}. See specifically Corollaries~\ref{cor:E} and~\ref{cor:E-inj}, Propositions~\ref{prop:hatUmonadic}, \ref{prop:KerF} and~\ref{prop:B-max}.
\end{proof}

%------------------------------------------------------------------------------
\goodbreak
\section{Modules, Serre ideals and quotients}
\label{se:MT}%
\medbreak
%------------------------------------------------------------------------------

Recall the Grothendieck category~$\cat{A}=\MT$ of Section~\ref{se:intro}.
\begin{Not}
The finitely presented objects $\cat{A}\fp$ are denoted by~$\mT$. The latter is the \emph{Freyd envelope of~$\cat{T}^c$} \cite[Chap.\,5]{Neeman01} and the (usual) Yoneda embedding $\yoneda\colon\cat{T}^c\hook \mT$, $x\mapsto \hat x=\Homcat{T^c}(-,x)$, identifies~$\cat{T}^c$ with the projective (and injective) objects in~$\mT$. Together with restricted-Yoneda of~\eqref{eq:res-yoneda}, we have the following commutative diagram
\[
\xymatrix{
\cat{T} \ar[r]^-{\yoneda}
& \MT = \cat{A} \kern-2em
\\
\cat{T}^c \topspace\  \ar@{^(->}[r]^-{\yoneda} \ar@{^(->}[u]
& \mT = \cat{A}\fp \,. \kern-3em \topspace \ar@{^(->}[u]
}
\]
The objects $\hat x$ remain projective but usually not injective in the whole module category~$\MT$. See Remark~\ref{rem:pure-inj} for injectives in~$\cat{A}$.
\end{Not}

\begin{Rem}
\label{rem:Sigma}%
The category $\cat{A}=\MT$ inherits a suspension
$\Sigma\colon\cat{A}\isoto \cat{A}$ from~$\cat{T}^c$ such that
$\Sigma\circ \yoneda = \yoneda \circ \Sigma$. On modules it
is~$\Sigma M=M\circ\Sigma\inv\colon(\cat{T}^c)\op\to\Ab$.
\end{Rem}

\begin{Rem}
\label{rem:hatG}%
The restricted-Yoneda functor $h\colon\cat{T}\to \MT$ is the universal co\-pro\-duct-preserving homological functor out of~$\cat{T}$. See~\cite[Cor.\,2.4]{Krause00}. In other words, for any coproduct-preserving homological functor $G\colon\cat T\to \cat{C}$ to a Grothendieck category~$\cat{C}$ there exists a unique \emph{exact} colimit-preserving functor $\hat G\colon\MT\to \cat{C}$ making the following diagram commute:
\[
\xymatrix{
\cat T \ar[r]^-{\yoneda} \ar[rd]_-{G}
& \MT \ar[d]^-{\hat G}
\\
& \cat{C}\,.
}
\]
\end{Rem}

\begin{Rem}
\label{rem:Day-1}%
The Day convolution product on~$\MT$ is the unique symmetric monoidal $\otimes\colon\MT\times\MT\to \MT$ such that $M\otimes-$ and $-\otimes N$ commute with colimits and such that $\hat x\otimes \hat y=\widehat{x\otimes y}$ for every $x,y\in\cat{T}^c$. Then restricted-Yoneda $h\colon\cat{T}\to \MT$ is a monoidal functor and $\hat X$ is flat in~$\MT$ for every~$X\in \cat{T}$, even non-compact (see~\cite[Prop.\,A.14]{BalmerKrauseStevenson17pp}). Also, $\hat x$ is rigid when~$x\in\cat{T}^c$, as every monoidal functor preserves rigid objects.

In fact the category~$\cat{A}=\MT$ is moreover closed, i.e.\ it admits an internal hom which we denote by $\ihomcat{A}$. This follows from a general fact about existence of adjoints in Grothendieck categories (Proposition~\ref{prop:closed}) or can be seen by another Day convolution argument. This internal hom functor~$\ihomcat{A}$ is characterized by the fact that $\ihomcat{A}(-,M)$ sends colimits to limits and the fact that $\ihomcat{A}(\hat x,-)=\widehat{\ihomcat{T}(x,-)}$ for every $x\in\cat{T}^c$, in the sense of Remark~\ref{rem:hatG}. Note that $\yoneda\colon\cat{T}\to \cat{A}$ needs not be a closed functor, outside of~$\cat{T}^c$.
\end{Rem}

\begin{Lem}\label{lem:compacttensor}
Let $x$ be an object of $\cat{T}^c$, with dual~$x^\vee=\ihom(x,\unit)$, and $M$ an object of~$\cat{A}=\MT$. Then there is a natural isomorphism of functors $(\cat{T}^c)\op\to \Ab$:
\begin{displaymath}
\hat{x}\otimes M \cong M(x^\vee \otimes -).
\end{displaymath}
In particular, we have a natural isomorphism $\widehat{\Sigma\unit} \otimes M \cong \Sigma M$ in~$\cat{A}$.
\end{Lem}

\begin{proof}
This is an immediate consequence of the description of the Day convolution product on $\MT$; the key observation is that it holds for representable functors $
\hat{x}\otimes \hat{y} \cong \widehat{x\otimes y} \cong \Homcat{K}(-,x\otimes y) \cong \Homcat{K}(x^\vee\otimes -, y)$. In particular, we have
\[
\widehat{\Sigma\unit} \otimes M \cong M((\Sigma\unit)^\vee \otimes -) \cong M(\Sigma\inv\unit \otimes -) \cong M\circ \Sigma\inv= \Sigma M \]
for every for $M\in\cat{A}$.
\end{proof}

\begin{Rem}
\label{rem:pure-inj}%
Although restricted-Yoneda $\yoneda\colon\cat{T}\to \cat{A}=\MT$ is in
general neither full nor faithful outside of~$\cat{T}^c$, the functor
$\yoneda$ restricts to an equivalence between the
pure-injective objects~$Y\in\cat{T}$ and the injective objects
of~$\cat{A}$. See~\cite[Cor.~1.9]{Krause00}. By definition, $Y$ is
\emph{pure-injective} if every pure mono $Y\to Z$ in $\cat{T}$ splits,
and  a morphism $Y\to Z$ is a \emph{pure mono} if the induced morphism
$\hat Y\to \hat Z$ is a monomorphism.
The interesting point is what happens on morphisms. We
even have slightly more than full-faithfulness. Indeed, the functor
$\yoneda$ induces an isomorphism
\begin{equation}
\label{eq:pure-h-ff}%
\Homcat{T}(X,Y)\isoto\Homcat{A}(\hat X,\hat Y)
\end{equation}
for $X\in\cat{T}$ arbitrary and $Y\in\cat{T}$ pure-injective. In particular, if $f\colon X\to Y$ is a \emph{phantom}, \ie $\hat f=0$ in~$\cat{A}$, with pure-injective target~$Y$ then $f=0$.
\end{Rem}

Injective objects in $\MT$ are also injective with respect to the internal hom:

\begin{Lem}\label{lem:ihominj}
Let $J$ be an injective object of~$\cat{A}=\MT$. Then the functor $\ihomcat{A}(-, J)\colon\cat{A}\op\to \cat{A}$ is exact.
\end{Lem}

\begin{proof}
For every $M\in\cat{A}$ and $c\in\cat{T}^c$, we have by Yoneda that $\Homcat{A}(\hat{c},M)\cong M(c)$. Hence the collection of functors $\{\Homcat{A}(\hat{c},-)\}_{c\in\cat{T}^c}$ detects exactness. On the other hand, $\Homcat{A}(\hat{c},\ihomcat{A}(-,J))\cong\Homcat{A}(\hat{c}\otimes-,J)$ and the functors $\hat{c}\otimes-\colon\cat{A}\to \cat{A}$ and $\Homcat{A}(-,J)\colon\cat{A}\op\to \Ab$ are exact by flatness of~$\hat{c}$ and injectivity of~$J$.
\end{proof}

\begin{Rem}
The content of the first part of the proof of the lemma is that the $\hat{c}$ for $c\in \cat{T}^c$ are, up to choosing a skeleton, a set of (finitely presented) projective generators for $\cat{A}$.
\end{Rem}

\begin{Prop}
\label{prop:phantom}%
Let $f\in\cat{T}$ be a phantom, \ie $\hat f=0$ in~$\cat{A}$.
\begin{enumerate}[\rm(a)]
\item
\label{it:phantom-a}%
For every~$X\in \cat{T}$, the morphism $X\otimes f$ is a phantom.
\item
\label{it:phantom-b}%
For every~$Y\in \cat{T}$ pure-injective, the morphism $\ihom(f,Y)$ is zero in~$\cat{T}$.
\item
\label{it:phantom-c}%
Let $Y$ be pure-injective and $X$ arbitrary. Then $\ihom(X,Y)$ is pure-injective.
\end{enumerate}
\end{Prop}

\begin{proof}
By Remark~\ref{rem:Day-1}, we have $\widehat{X\otimes f}=\hat X\otimes \hat f=0$, hence~\eqref{it:phantom-a}. This implies that $X\otimes-$ preserves the class of exact triangles whose third map is a phantom. In words, this means that every $X\otimes-$ is pure-exact. Consequently, the functor
\[
\Homcat{T}(-,\ihom(X,Y))\cong \Homcat{T}(X\otimes-,Y)\cong \Homcat{T}(-,Y)\circ (X\otimes-)
\]
is the composite of two pure-exact functors if~$Y$ is pure-injective. This gives~\eqref{it:phantom-c}. Finally, to check~\eqref{it:phantom-b}, for every $X\in\cat{T}$, we have
\[
\Homcat{T}(X,\ihom(f,Y))\simeq\Homcat{T}(X\otimes f,Y)\simeq\Homcat{A}(\widehat{X\otimes f},\hat Y)=0
\]
where the second isomorphism holds by~\eqref{eq:pure-h-ff} and the vanishing by~\eqref{it:phantom-a}. This shows that $\ihom(f,Y)=0$ by Yoneda.
\end{proof}

\begin{center}
*\ *\ *
\end{center}

We are interested in the Serre subcategories~$\cat{B}$ of~$\mT$ and~$\MT$. We focus on the $\otimes$-ideals meaning of course $M\otimes\cat{B}\subseteq\cat{B}$ for every~$M$ in the ambient category.

\begin{Conv}
\label{conv:Sigma}%
All Serre subcategories of~$\MT$ that we consider are assumed stable under suspension (Remark~\ref{rem:Sigma}). For  $\otimes$-ideals it follows from Lemma~\ref{lem:compacttensor} that we can safely omit this condition and we shall do so from now on.
\end{Conv}

The $\otimes$-ideal condition can be tested just using the finitely presented projectives:

\begin{Lem}
\label{lem:id-Tc}%
Let $\cat{B}$ be a Serre subcategory of $\mT$ which is closed under tensoring with finitely presented representable functors, i.e.\ closed under the action of $\cat{T}^c$ under the Yoneda embedding. Then $\cat{B}$ is a Serre $\otimes$-ideal in~$\mT$.
\end{Lem}

\begin{proof}
For every $M\in \mT$, there is an epimorphism $\hat x\onto M$ with~$x\in\cat{T}^c$. Tensoring with any $N \in \cat{B}$ we get $\hat{x}\otimes N \onto M\otimes N$ and $\hat x\otimes N$ belongs to~$\cat{B}$ by assumption, hence so does~$M\otimes N$ since $\cat{B}$ is Serre.
\end{proof}

We collect a few facts about the quotients of~$\cat{A}=\MT$ by a Serre $\otimes$-ideal.

\begin{Prop}
\label{prop:A/B}%
Let $\cat{B}\subset \cat{A}\fp$ be a Serre $\otimes$-ideal. Let $\overrightarrow{\cat{B}}\subset \cat{A}$ the localizing (Serre) $\otimes$-ideal it generates; see~\eqref{eq:B^to}.
\begin{enumerate}[\rm(a)]
\item
The Grothendieck-Gabriel quotient of the small categories $\cat{A}\fp/\cat{B}$ maps fully faithfully into the quotient of the big ones~$\bat{A}:=\cat{A}/\overrightarrow{\cat{B}}$ and identifies the former with the finitely presented objects of the latter:
\begin{equation}
\label{eq:quotient}%
\vcenter{
\xymatrix@R=2em{
\cat{B} \ar@{}[r]|-{\subseteq} \ar@{ >->}[d]
& \overrightarrow{\cat{B}} \kern-1em \ar@{ >->}[d]
\\
\kern-4em \mT = \cat{A}\fp \ar@{}[r]|-{\subseteq} \ar@{->>}[d]_-{Q}
& \cat{A} = \MT \kern-4.5em \ar@{->>}[d]_-{Q}
\\
\kern-3.5em \cat{A}\fp/\cat{B} = \bat{A}\fp \ar@{}[r]|-{\subseteq}
& \bat{A} := \cat{A}/\overrightarrow{\cat{B}} \kern-4em \ar@<-.5em>[u]_-{R}
}}
\end{equation}
\item
The functor~$Q$ is universal among exact functors out of~$\cat{A}$ with kernel~$\overrightarrow{\cat{B}}$. So $Q$ is exact and any functor $G$ out of~$\bat{A}$ is exact if and only if $GQ$ is exact.
\item
The quotient category~$\bat{A}$ inherits a unique tensor structure such that $Q\colon\cat{A}\to\bat{A}$ is a tensor-functor. This tensor remains colimit-preserving in each variable and admits an internal hom functor $\ihom_{\bat{A}}$.
\item
The functor $Q$ preserves flat objects.
\item
The functor~$Q$ admits a right adjoint on the big categories~$R\colon\bat{A}\to \cat{A}$, which is left exact, preserves injective objects and satisfies $QR\cong\Id_{\bat{A}}$. Moreover, $R$ preserves internal homs
\begin{equation}
\label{eq:aux-1}%
R (\ihom_{\bat{A}}(Y_1,Y_2)) \cong \ihomcat{A}(RN_1,RN_2).
\end{equation}
\end{enumerate}
\end{Prop}

\begin{proof}
All this is standard abelian category theory. See Appendix~\ref{app:Grothendieck}.
\end{proof}

\begin{Not}
When the category~$\cat{B}\subseteq\mT$ is clear from the context, we shall often denote the composite~$\boneda:=Q\circ\yoneda\colon\cat{T}\to \bat{A}$ by the simple notation
\begin{equation}
\label{eq:bar}%
\bar X=Q\yoneda(X)=Q(\hat X)
\qquadtext{and}
\bar f=Q\yoneda(f)=Q(\hat f)
\end{equation}
for every object~$X$ and morphism~$f$ in~$\cat{T}$.
\end{Not}

The above proposition holds for any locally coherent Grothendieck category with a tensor. In the particular case of~$\cat{A}=\MT$, we have the following consequences.

\begin{Cor}
\label{cor:A/B}%
Let $\cat{B}\subset\cat{A}\fp=\mT$ be a Serre $\otimes$-ideal and $\overrightarrow{\cat{B}}\subset\cat{A}=\MT$ as in Proposition~\ref{prop:A/B}. Then:
\begin{enumerate}[\rm(a)]
\smallbreak
\item
\label{it:boneda}%
The functor $\boneda=Q\circ\yoneda\colon\cat{T}\to\bat{A}=\MT/\overrightarrow{\cat{B}}$ is homological coproduct-preserving and monoidal.
\smallbreak
\item
\label{it:bar-flat}%
For every $X\in\cat{T}$, the object $\bar X$ remains flat in~$\bat{A}$.
\smallbreak
\item
\label{it:inj-bar}%
Every injective object in~$\bat{A}$ is of the form $\bar E$ for a unique pure-injective~$E\in \cat{T}$ with an isomorphism $\hat E\simeq R(\bar E)$ in~$\cat{A}$. Moreover, for every object $X\in\cat{T}$, the functor $\boneda$ induces an isomorphism
\begin{equation}
\label{eq:hom-to-bar-E}%
\Homcat{T}(X,E)\cong\Homcat{{}\bat{A}}(\bar X,\bar E).
\end{equation}
\smallbreak
\item
\label{it:inj-flat}%
Any injective object in the category~$\bat{A}$ is flat.
\end{enumerate}
\end{Cor}

\begin{proof}
We use everywhere the results of Proposition~\ref{prop:A/B}. Part~\eqref{it:boneda} holds since~$Q$ is exact and coproduct-preserving and monoidal. Part~\eqref{it:bar-flat} holds since $Q$ preserves flat objects. Combining Remark~\ref{rem:pure-inj} and $QR\cong\Id$, we obtain the first sentence of~\eqref{it:inj-bar}. It remains to prove~\eqref{eq:hom-to-bar-E}. It is the following composite of isomorphisms:
\[
\Homcat{T}(X,E)\cong\Homcat{A}(\hat X,\hat E)\cong \Homcat{A}(\hat X,R\bar E)\cong\Homcat{{}\bat{A}}(\bar X,\bar E)
\]
using~\eqref{eq:pure-h-ff}, the defining relation $\hat E\cong R\bar E$ and the $Q\adj R$ adjunction. Finally~\eqref{it:inj-flat} results from~\eqref{it:inj-bar} and~\eqref{it:bar-flat}.
\end{proof}

\begin{Rem}
The analogue of Lemma~\ref{lem:ihominj} also holds in $\bat{A}$, namely, if $J$ is an injective object of $\bat{A}$ then it is also injective for the internal hom i.e.\ the functor $\ihom_{\bat{A}}(-,J)$ is exact. To see this, as $\ihom_{\bat{A}}(-,J)$ is a right adjoint, we only need to prove that it is right exact; keep in mind that it is \emph{contravariant}. Since $R$ preserves injectives (Proposition~\ref{prop:A/B}), the object $RJ$ is injective in~$\cat{A}$. By~\eqref{eq:aux-1} and $QR\cong\Id$, our functor $\ihom_{\bat{A}}(-,J)$ is the following composition of functors:
\[
 Q\circ \ihomcat{A}(-,RJ)\circ R\,.
\]
The first one, $R$, is \emph{left} exact (Proposition~\ref{prop:A/B}). The second one is \emph{contravariant} and exact by Lemma~\ref{lem:ihominj} since $RJ$ is injective. The third one is exact. So the composite is \emph{right} exact, as desired.
\end{Rem}

%------------------------------------------------------------------------------
\goodbreak
\section{Constructing pure-injectives}
\label{se:E}%
\medbreak
%------------------------------------------------------------------------------

For this section, we fix a Serre $\otimes$-ideal $\cat{B}\subset\cat{A}\fp=\mT$ of finitely presented $\cat{T}^c$-modules. As in Section~\ref{se:MT}, we denote by~$\cat{A}=\MT$ the whole category of $\cat{T}^c$-modules and by~$Q\colon\cat{A}\to \bat{A}=\MT/\overrightarrow{\cat{B}}$ the corresponding quotient. We write $\boneda\colon\cat{T}\to \bat{A}$, or~$X\mapsto\bar{X}$, for the composed functor~$Q\circ\yoneda$, see~\eqref{eq:bar}.

\begin{Cons}
\label{cons:E}%
Consider an injective envelope of the $\otimes$-unit~$\bar \unit$
\begin{equation}
\label{eq:bar-E}%
\bar\eta\colon\bar\unit\into\bar E
\end{equation}
in the Grothendieck category~$\bat{A}$. By Corollary~\ref{cor:A/B}\,\eqref{it:inj-bar} there exists a pure-injective
\[
\qquad
E=E(\cat{B})\quadtext{in}\cat{T},
\]
unique, up to unique isomorphism, together with an isomorphism $\hat E\cong R(\bar E)$ and therefore $Q(\hat E)\cong\bar E$. The identification~\eqref{eq:hom-to-bar-E} gives us a unique morphism $\eta\colon\unit\to E$ such that $Q(\hat \eta)=\bar \eta$. This justifies the notation~$\bar\eta$ and~$\bar{E}$ in~\eqref{eq:bar-E}.
\end{Cons}

Let us investigate the properties of the pure-injective object~$E=E(\cat{B})$ of Construction~\ref{cons:E}. We already know by Corollary~\ref{cor:A/B}\,\eqref{it:inj-flat} that $\bar E$ is flat and clearly $\bar E$ is non-zero in~$\bat{A}$ as soon as $\cat{B}\subset\cat{A}$ is proper, since $\bar\unit$ is a subobject of~$\bar E$.

\begin{Prop}
\label{prop:E-tens}%
Consider an exact triangle~$\Delta$ in~$\cat{T}$ on the morphism $\eta\colon\unit\to E$:
\[
\xymatrix{
\Delta\colon
&& \Sigma\inv W \ar[r]^-{\xi}
& \unit \ar[r]^-{\eta}
& E \ar[r]^-{\zeta}
& W.
}
\]
It satisfies the following properties:
\begin{enumerate}[\rm(a)]
\smallbreak
\item
The image of~$\Delta$ in~$\bat{A}=\MT/\overrightarrow{\cat{B}}$ is the following exact sequence
\[
\xymatrix{
0 \ar[r] & {\bar{\unit}} \ar[r]^-{\bar{\eta}} & \bar{E} \ar[r]^-{\bar{\zeta}} & \bar{W} \ar[r] &  0\,.
}
\]
In particular, $\bar \xi=Q(\hat\xi)$ is zero.
\smallbreak
\item
Every morphism $f\colon X\to \unit$ in~$\cat{T}$ whose image $\bar f=0$ vanishes in~$\bat{A}$ factors (possibly non-uniquely) through~$\xi$, that is, $f=\xi\,\tilde f$ for $\tilde f\colon X\to \Sigma\inv W$.
\smallbreak
\item
We have $\xi\otimes E=0$.
\smallbreak
\item Let $f\colon X\to c$ be a morphism in~$\cat{T}$ such
    that $\bar f=0$ in~$\bat{A}$ and $c\in \cat{T}^c$ is
    compact. Then $E\otimes f=0$.
\end{enumerate}
\end{Prop}

\begin{proof}
Part~(a) is immediate from the fact that $\boneda\colon\cat{T}\to \bat{A}$ is homological and the fact that $\boneda(\eta)=\bar \eta$ is a monomorphism~\eqref{eq:bar-E}. To prove~(b), let $f\colon X\to \unit$ in~$\cat{T}$ such that $\bar f=0$. By exactness of~$\Delta$, we only need to prove that the composite $\eta\,f\colon X\to E$ vanishes in~$\cat{T}$. By pure-injectivity of~$E$, this follows from~\eqref{eq:hom-to-bar-E} and the vanishing $\overline{\eta\,f}=\bar\eta\bar f=0$ in~$\bat{A}$. We also deduce~(c) from the same isomorphism~\eqref{eq:hom-to-bar-E} in Corollary~\ref{cor:A/B} applied to the morphism $\xi\otimes \id_E$ from $X:=\Sigma\inv W\otimes E$ to~$E$ and the already established~$\bar \xi=0$ in~(a). For~(d), consider the morphism $g\in\Homcat{T}(c^\vee\otimes X,\unit)\cong\Homcat{T}(X,c)\ni f$ corresponding to~$f$ under the adjunction. This uses rigidity of~$c$. Explicitly, $g$ is the following composite:
\[
\xymatrix@C=4em{
c^\vee \otimes X \ar[r]^-{1\otimes f}
& c^\vee \otimes c \ar[r]^-{\ev}
& \unit\,.
}
\]
In particular, since $\boneda\colon\cat{T}\too \bat{A}$ is monoidal, we see that $\bar g=0$. By~(b) applied to~$g$, we see that $g$ factors through~$\xi$ and thus by~(c), we have $g\otimes E=0$.
We can then recover $f$ from~$g$ as the composite
\[
\xymatrix@C=4em{
X \ar[r]^-{\coev\otimes 1}
& c\otimes c^\vee \otimes X\ar[r]^-{1\otimes g}
& c
}
\]
In particular $g\otimes E=0$ implies $f\otimes E=0$ as well.
\end{proof}

The following lemma formalizes an elementary argument that we will use numerous times.

\begin{Lem}\label{lem:square}
Suppose we are given maps
\begin{displaymath}
f\colon X\to Y \quad \text{and} \quad f'\colon X'\to Y'
\end{displaymath}
such that $Y\otimes f'$ is a monomorphism and $f\otimes Y' =0$. Then $f\otimes X' = 0$.
\end{Lem}
\begin{proof}
We can arrange the various maps between tensor products in a square
\[
\xymatrix@C=4em{
X\otimes X' \ar[r]^-{f\otimes X'} \ar[d]_-{X\otimes f'}
& Y\otimes X' \ar[d]^-{Y\otimes f'}
\\
X\otimes Y' \ar[r]^-{f\otimes Y'}
& Y\otimes Y'
}
\]
which commutes by bifunctoriality of the tensor product. By hypothesis
the bottom horizontal map vanishes, and thus so do both composites. As
the rightmost vertical map is a monomorphism this forces $f\otimes
X'=0$ as claimed.
\end{proof}

\begin{Thm}
\label{thm:E-tens}%
With the notation of Construction~\ref{cons:E} for $E=E(\cat{B})$, we have in~$\cat{A}$
\[
\overrightarrow{\cat{B}}=\Ker(\hat E\otimes-).
\]
\end{Thm}

\begin{proof}
Let $M\in\cat{A}$ be such that $\hat E\otimes M=0$. There exists a morphism $f\colon X\to Y$ in~$\cat{T}$ such that $M$ is the image of~$\hat f$. Indeed, we may take a projective presentation of $M$ using objects in the image of $\Add(\cat{T}^c)$ and then take the cone on the map giving the presentation. The assumption that $\hat E\otimes M=0$ implies $\bar E\otimes \bar f=0$ in~$\bat{A}$.

Applying Lemma~\ref{lem:square} to $\bar f\colon \bar X\to \bar Y$ and $\bar \eta\colon \bar \unit\to \bar E$, which is reasonable since $\bar \eta$ is a monomorphism and $\bar Y$ is flat, we see $\bar f=0$. Therefore, by exactness of~$Q$, we have $Q(M)=\im(\bar f)=0$, which proves that $M\in\Ker(Q)=\overrightarrow{\cat{B}}$.

Conversely, let us show that $\overrightarrow{\cat{B}}\subseteq\Ker(\hat E\otimes-)$. As $\hat E$ is flat in~$\cat{A}$, it suffices to prove the inclusion of the finitely presented part $\cat{B}\subseteq \Ker(\hat E\otimes-)$. Let $M\in\cat{B}$. Then there exists a morphism $f\colon x\to y$ in~$\cat{T}^c$ such that $M$ is the image of~$\hat f$. By Proposition~\ref{prop:E-tens}\,(d), we know that $\bar f=0$ in~$\bat{A}$ forces $E\otimes f=0$ in~$\cat{T}$ and therefore $\hat E\otimes M=0$ in~$\cat{A}$ as wanted.
\end{proof}

\begin{Cor}
\label{cor:E-residual}%
Let $f\colon X\to Y$ be a morphism in~$\cat{T}$. Then we have $E\otimes f=0$ in~$\cat{T} \implies \hat E\otimes \hat f=0$ in~$\cat{A}$ (\ie $E\otimes f$ is a phantom) $\iff \bar f=0$ in~$\bat{A}$. Moreover, all three properties are equivalent if~$Y$ is in~$\cat{T}^c$.
\qed
\end{Cor}

\begin{Cor}
\label{cor:E-tens}%
Let $X\in\cat{T}$. Then $\bar X=0$ in~$\bat{A}$ if and only if~$E\otimes X=0$.
\end{Cor}

\begin{proof}
This follows from the previous corollary for $f=1\colon X\to X$ since restricted-Yoneda~$X\mapsto \hat X$ is conservative: $\hat E\otimes\hat X=0\implies E\otimes X=0$.
\end{proof}

\begin{Rem}
It follows from the above discussion that a Serre $\otimes$-ideal $\cat{B}$ in $\mT$ is determined by each of the following:
\begin{enumerate}[(1)]
\item The morphisms $f\colon c\to \unit$ between compact objects in
  $\cat{T}$ such that $\bar f=0$.
\item The morphisms $f\colon c\to d$ between compact objects in
  $\cat{T}$ such that $\bar f=0$.
\item The pure-injective object $E=E(\cat{B})$ in $\cat{T}$.
\end{enumerate}
\end{Rem}

In the spirit of \cite[Theorem~3.10]{BalmerKrauseStevenson17pp} we can also give a generator for $\overrightarrow{\cat{B}}$.

\begin{Prop}\label{prop:I-gens}
Let $\cat{B}$ be a Serre $\otimes$-ideal in~$\mT$ and let $I=\Ker(\hat\eta)$ be the kernel in~$\MT$ of the map $\eta\colon\unit\to E$ of Construction~\ref{cons:E}. Then we have
\begin{displaymath}
\overrightarrow{\cat{B}} = \langle I \rangle
\end{displaymath}
i.e.\ $\overrightarrow{\cat{B}}$ is the smallest Serre $\otimes$-ideal containing $I$ and closed under coproducts.
\end{Prop}

\begin{proof}
By Construction~\ref{cons:E}, we know that $Q(\hat\eta)=\bar\eta\colon\bar\unit\into \bar{E}$ is a monomorphism. Since the quotient functor $Q$ is exact we see that $QI\cong 0$ and so $I$ lies in $\Ker(Q)=\overrightarrow{\cat{B}}$. As $\overrightarrow{\cat{B}}$ is a localizing Serre $\otimes$-ideal it follows that $\langle I \rangle \subseteq \overrightarrow{\cat{B}}$.

On the other hand, suppose $M$ is in $\overrightarrow{\cat{B}}$ which is, by Theorem~\ref{thm:E-tens}, equivalent to saying $M\otimes \hat{E} \cong 0$. Choose some monomorphism $\alpha\colon M\to Y$ with $Y$ flat (for instance an injective envelope) and consider the diagram
\begin{displaymath}
\xymatrix@R=2em{
&& M \ar[r]^-{M\otimes \hat{\eta}} \ar@{ >->}[d]^-{\alpha} \ar@{-->}[ld]
& M\otimes \hat{E} =0 \kern-2em \ar[d]^-{\alpha\otimes \hat{E}}
\\
0 \ar[r]
& Y\otimes I \ar[r]
& Y \ar[r]^-{Y\otimes \hat{\eta}}
& Y \otimes \hat{E}
}
\end{displaymath}
where the bottom row is exact by flatness of $Y$ and the dashed arrow exists by the universal property of the kernel and is a monomorphism since $\alpha$ is one. This exhibits $M$ as a subobject of $Y\otimes I$ and hence $M\in \langle I \rangle$. Thus $\overrightarrow{\cat{B}} \subseteq \langle I \rangle$.
\end{proof}

Next let us remark on a little extra structure that $E$ can be endowed with. It follows from Proposition~\ref{prop:E-tens}\,(c) that $\Delta\otimes E$ splits, \ie that $E\otimes E\simeq E\oplus (W\otimes E)$, and in particular there is a retraction $E\otimes E\to E$ of~$\eta\otimes E$. Let us be more precise:
\begin{Prop}
There exists a morphism $\mu\colon E\otimes E\to E$ in~$\cat{T}$ such that
\begin{equation}
\label{eq:two-side-unit}%
\mu\circ(\eta\otimes 1)=\id_E=\mu\circ(1\otimes\eta)\,.
\end{equation}
\end{Prop}

\begin{proof}
Using flatness of the injective~$\bar E$, we have a split exact sequence in~$\bat{A}$
\[
\xymatrix@C=1.5em{
0 \ar[r]
& \bar E \ar[rr]^-{\bar \eta\otimes 1}
&& \bar E\otimes \bar E \ar[rr]^-{\bar \zeta\otimes 1}
&& \bar W\otimes \bar E \ar[r]
& 0\,.}
\]
Choose a retraction of~$\bar\eta\otimes1$, say $\bar\mu_0\colon\bar E\otimes \bar E\to \bar E$. Consider the endomorphism $\varphi:=\bar\mu_0\circ(1\otimes \bar\eta)$ of~$\bar E$, that we wish was equal to the identity. At least it satisfies
\begin{equation}
\label{eq:aux-2}%
\varphi\,\bar\eta=\bar\mu_0\,(1\otimes \bar\eta)\,\bar\eta=\bar\mu_0\,(\bar\eta\otimes \bar\eta)\,=\bar\mu_0\,(\bar\eta\otimes 1)\,\bar\eta=\bar\eta.
\end{equation}
Hence, as $\bar\eta\colon\bar\unit\to \bar E$ was an injective envelope, $\varphi$ is an isomorphism. Let now $\bar \mu:=\bar\mu_0\circ (\varphi\inv\otimes1)\colon \bar E\otimes \bar E\to \bar E$. Direct computation gives us that $\bar \mu$ remains a retraction of~$\bar\eta\otimes 1$, since $\varphi\inv\bar\eta=\bar\eta$ by~\eqref{eq:aux-2}. But now $\bar\mu$ also satisfies $\bar\mu\,(1\otimes\bar\eta)=\id_{\bar E}$ as the outer commutativity of the following diagram shows:
\[
\xymatrix@C=4em{
\bar E \ar[d]_-{1\otimes \bar \eta} \ar[r]^-{\varphi\inv}
& \bar E \ar[d]_-{1\otimes \bar \eta} \ar[rd]^-{\varphi}
\\
\bar E \otimes \bar E \ar[r]^-{\varphi\inv\otimes 1} \ar@/_2em/[rr]^-{\bar\mu}
& \bar E \otimes \bar E \ar[r]^-{\bar\mu_0}
& \bar E\,.
}
\]

Finally, thanks to~\eqref{eq:hom-to-bar-E}, we can lift $\bar\mu$ to a unique morphism $\mu\colon E\otimes E\to E$ such that $\bar\mu = Q(\hat\mu)$ and~\eqref{eq:two-side-unit} holds, as all the relevant maps end in~$E$.
\end{proof}

\begin{Rem}
Unfortunately, we do not know how to upgrade $\mu\colon E\otimes E\to E$ to an associative, or commutative, ring structure.
\end{Rem}

Let us end the section with a rather amusing application of Proposition~\ref{prop:E-tens}; we will show that one can characterize phantoms ending in a compact object in terms of tensoring with pure-injectives. First, a preparatory lemma which could be of independent interest.

\begin{Lem}\label{lem:phantom}
Suppose $f\colon x\to c$ is a morphism between compacts and let $E(c)$ denote the pure-injective envelope of $c$. Then
\begin{displaymath}
f\otimes E(c) = 0 \implies f\otimes c = 0 \implies f =0.
\end{displaymath}
\end{Lem}
\begin{proof}
We are in the situation of Lemma~\ref{lem:square}: we have maps $\hat f\colon \hat x \to \hat c$ and $\hat \eta\colon \hat c \into \widehat{E(c)}$ and we know that $\hat c\otimes \hat \eta$ is a monomorphism, since $\hat c$ is flat. Thus if $\hat f\otimes E(c)$ vanishes so does $\hat f \otimes \hat c$.  It follows that $f\otimes c$ is trivial as $x\otimes c$ is compact.

Now suppose $f\otimes c$ is zero. Then we can apply Lemma~\ref{lem:square} again to the maps $f\colon x\to c$ and $\coev_c \colon \unit \to c^\vee \otimes c$, since $c\otimes\coev_c$ is a split monomorphism, to deduce that $f=0$.
\end{proof}

\begin{Cor}\label{Cor:phantom}
Let $c$ be a compact object of $\cat{T}$, denote by $E(c)$ the pure-injective envelope of~$c$, and let $f\colon X\to c$ be a morphism. Then $E(c)\otimes f = 0$ if and only if $f$ is phantom.
\end{Cor}
\begin{proof}
We first show that if $f$ is phantom then $E(c)$ witnesses this. Consider Proposition~\ref{prop:E-tens} for $\cat{B} = 0$. We have an exact triangle
\begin{displaymath}
\xymatrix{
\Sigma\inv W(c) \ar[r]^-p & c \ar[r] & E(c) \ar[r] & W(c)
}
\end{displaymath}
where $c\to E(c)$ is a pure-injective envelope of $c$ and $p$ is the (weakly) universal phantom with target~$c$. The proposition tells us that $E(c)\otimes p = 0$. The statement then follows as any phantom $f\colon X\to c$ factors via $p$.

Now we show that if $f\otimes E(c)$ is zero then $f$ is a phantom. To this end, suppose $f\otimes E(c)=0$ but $f$ is not phantom. Then there is an $x\in \cat T^c$ and a non-vanishing composite
\[
\xymatrix{
x \ar[r]^-g & X \ar[r]^-f & c.
}
\]
Applying $-\otimes E(c)$ we see that $fg\otimes E(c)=0$, since $f\otimes E(c)=0$, but this contradicts Lemma~\ref{lem:phantom}.
\end{proof}

%------------------------------------------------------------------------------
\goodbreak
\section{Maximal Serre tensor ideals}
\label{se:max}%
\medbreak
%------------------------------------------------------------------------------

We now want to isolate interesting Serre $\otimes$-ideals of~$\cat{B}\subset\mT=\cat{A}\fp$, to which we can apply the constructions of the previous sections.

\begin{Prop}
\label{prop:max}%
There exists a Serre $\otimes$-ideal $\cat{B}$ in~$\cat{A}\fp$ which is maximal with respect to inclusion among those which do not intersect~$\SET{\hat c}{c\in\cat{T}^c,\ c\neq0}$.
\end{Prop}

\begin{proof}
This is immediate by Zorn since the union of a tower of Serre $\otimes$-ideals which do not meet a given class of objects still has this property.
\end{proof}

\begin{Def}
\label{def:max}%
We call $\cat{B}\subset \cat{A}\fp$ a \emph{$\cat{T}^c$-maximal Serre $\otimes$-ideal} if it is maximal among those such that $\cat{B}\cap\yoneda(\cat{T}^c)=0$, as in Proposition~\ref{prop:max}.
\end{Def}

\begin{Rem}
\label{rem:nonunique}%
We have used Zorn's lemma to guarantee the existence of \emph{a} $\cat{T}^c$-maximal Serre $\otimes$-ideal, but of course it need not be unique. We will prove, see Theorem~\ref{thm:unique} and Corollary~\ref{cor:unique}, that under certain circumstances there is a unique such ideal. However, this is not the general expectation: see Remark~\ref{rem:manyfields}.
\end{Rem}

\begin{Def}
\label{def:local}%
Recall from~\cite[\S\,4]{Balmer10b} that the rigid tt-category $\cat{T}^c$ is called \emph{local} if $x\otimes y=0$ forces $x=0$ or $y=0$. Conceptually, it means that the spectrum~$\Spc(\cat{T}^c)$ is a local topological space. It has a unique closed point, namely~$\gm=(0)$. By extension, we shall say that $\cat{T}$ is \emph{local} when $\cat{T}^c$ is. Note that $\cat{T}$ itself will almost never satisfy the property that $X\otimes Y=0\implies X=0\textrm{ or }Y=0$, as one can easily see with Rickard idempotents for instance.
\end{Def}

\begin{Rem}
\label{rem:local}%
To every point $\cat P\in\Spc(\cat{T}^c)$ of the tt-spectrum of~$\cat{T}^c$ we can associate a \emph{local} category $\cat{T}_{\cat P}:=\cat{T}/\Loc(\cat{P})$. Its rigid-compacts $(\cat{T}_{\cat{P}})^c$ coincide, by a result of Neeman~\cite{Neeman92b}, with the idempotent-completion of the Verdier localization~$\cat{T}^c/\cat P$. This construction extends the algebro-geometric one, in the sense that if we start with the derived category $\cat{T}=\Der(X):=\Der_{\Qcoh}(\calO_X)$ of a quasi-compact and quasi-separated scheme, so that $\cat{T}^c=\Dperf(X)$ is the derived category of perfect complexes, and if $\cat P=\cat P(x)$ is the prime corresponding to a point $x\in X$ under the homeomorphism $X\simeq \Spc(\Dperf(X))$, then the local category $\cat{T}_{\cat P}$ is naturally equivalent to the derived category~$\Der(\calO_{X,x})$ over the local ring~$\calO_{X,x}$ at~$x$.
\end{Rem}

\begin{Hyp}
\label{hyp:loc-max}%
Assume $\cat{T}^c$ local and $\cat{B}\subset\cat{A}\fp$ a $\cat{T}^c$-maximal Serre $\otimes$-ideal.
\end{Hyp}

\begin{Prop}
\label{prop:max-saturated}%
Under Hypothesis~\ref{hyp:loc-max}, let $M\in \cat{A}\fp$ be such that $M\otimes \hat d\in \cat{B}$ for some non-zero $d\in\cat{T}^c$. Then we have $M\in\cat{B}$.
\end{Prop}

\begin{proof}
Consider $\cat{B}':=\SET{M\in \cat{A}\fp}{M\otimes \hat d\in \cat{B},\textrm{ for some non-zero }d\in\cat{T}^c}$. Since every $\hat d$ is flat in~$\cat{A}\fp$ (Remark~\ref{rem:Day-1}), it is easy to verify that $\cat{B}'$ is a Serre $\otimes$-ideal\textemdash{}this is immediate from the corresponding properties for~$\cat{B}$. Since~$\cat{T}$ is local, we know that for $\hat c \neq 0$ the object $\hat c\otimes \hat d\cong \widehat{c\otimes d}$ remains non-zero, and so $\cat{B}'$ still avoids $\SET{\hat c}{c\in \cat{T}^c,\ c\neq 0}$. We then conclude by maximality of~$\cat{B}$ that $\cat{B}=\cat{B}'$.
\end{proof}

\begin{Lem}
\label{lem:max-rigid}%
Under Hypothesis~\ref{hyp:loc-max}, let $c\in\cat{T}^c$ be non-zero. Consider
\[
 \coev_c\colon\unit\to c^\vee\otimes c
\qquadtext{and}
\ev_c\colon c\otimes c^\vee\to \unit
\]
the unit and the counit of the $\otimes\adj\ihom$ adjunction, where $c^\vee=\ihom(c,\unit)$ denotes the dual in~$\cat{T}^c$ as before. Then the images of these morphisms in~$\bat{A}\fp=\cat{A}\fp/\cat{B}$ are a monomorphism $\bar \unit \into \overline{c^\vee}\otimes\bar c$ and an epimorphism $\bar c\otimes\overline{c^\vee}\onto\bar \unit $, respectively.
\end{Lem}

\begin{proof}
By the unit-counit relations, the morphism $1_{\bar c}\otimes\overline{\coev_c}$ is a split monomorphism (already in~$\cat{T}^c)$. Hence, by flatness of~$\bar c$, we have $\bar c\otimes\Ker(\overline{\coev_c})=0$. By Proposition~\ref{prop:max-saturated} we deduce $\Ker(\overline{\coev_c})=0$. The other one is similar.
\end{proof}

\begin{Prop}
\label{prop:all-generate}%
Under Hypothesis~\ref{hyp:loc-max}, we have:
\begin{enumerate}[\rm(a)]
\smallbreak
\item
The homological $\otimes$-functor $\boneda\colon\cat{T}\to \bat{A}=\MT/\overrightarrow{\cat{B}}$ is conservative on compacts, that is, every non-zero $c\in\cat{T}^c$ has non-zero image $\bar c\neq 0$ in~$\bat{A}\fp$, or equivalently it detects isomorphisms between compact objects.
\smallbreak
\item
For every non-zero $c\in \cat{T}^c$, the Serre $\otimes$-ideal $\ideal{\bar c}$ generated by~$\bar c$ in~$\bat{A}\fp$ is the whole $\bat{A}\fp$.
\smallbreak
\item
If $c\in\cat{T}^c$ is non-zero and $f\colon X\to Y$ in~$\cat{T}$ is such that $c\otimes f=0$, then $\bar f=0$.
\end{enumerate}
\end{Prop}

\begin{proof}
Part~(a) is immediate from $\hat c\notin\cat{B}$ for all non-zero $c\in\cat{T}^c$. Detection of isomorphism then follows by applying the homological functor $c\mapsto \bar c$ to the cone of a morphism in~$\cat{T}^c$. Part~(b) is immediate from Lemma~\ref{lem:max-rigid} and so is~(c) by an application of Lemma~\ref{lem:square} to $\bar f$ and $\coev_c \colon \unit \to c^\vee\otimes c$, using that $\overline{\coev_c}\otimes 1$ is a monomorphism by Lemma~\ref{lem:max-rigid} and flatness of~$\bar Y$.
\end{proof}

From this we deduce that although we only assumed~$\cat{B}$ maximal among those subcategories meeting $\cat{T}^c$ trivially, it is automatically plain maximal in~$\cat{A}\fp$.

\begin{Cor}
\label{cor:all-generate}%
Under Hypothesis~\ref{hyp:loc-max}, we have:
\begin{enumerate}[\rm(a)]
\item
The subcategory $\cat{B}$ is a maximal proper Serre $\otimes$-ideal of~$\cat{A}\fp$.
\item
The only Serre $\otimes$-ideals of~$\bat{A}\fp$ are zero and~$\bat{A}\fp$.
\item
Every non-zero object of $\bat{A}\fp$ generates $\bat{A}\fp$ as a Serre $\otimes$-ideal, and generates $\bat{A}$ as a localizing $\otimes$-ideal.
\end{enumerate}
\end{Cor}

\begin{proof}
  Let $\cat{B}\subsetneq\cat{C}\subseteq\cat{A}\fp$ be a Serre
  $\otimes$-ideal and $\bat{C}=Q(\cat{C})\neq 0$ the corresponding
  Serre $\otimes$-ideal of the quotient~$\bat{A}\fp$. By maximality
  of~$\cat{B}$ among the Serre $\otimes$-ideals of~$\cat{A}\fp$
  avoiding $\SET{\hat c}{c\in\cat{T}^c,\ c\neq0}$
  (Definition~\ref{def:max}), there exists $c\in\cat{T}^c$ non-zero
  such that $\hat c\in \cat C$, that is, $\bar c\in\bat C$. We
  conclude by Proposition~\ref{prop:all-generate}\,(b) that $\bat{C}$
  is the whole category~$\bat{A}\fp$. Hence $\cat{C}=Q\inv(\bat{C})$
  is the whole of~$\cat{A}\fp$. Therefore $\overrightarrow{\bat{C}}$
  is the whole of~$\overrightarrow{\bat{A}\fp}=\bat{A}$ and all the statements
  follow.
\end{proof}

We now have the following consequences:
\begin{Cor}
\label{cor:bar-0}%
Under Hypothesis~\ref{hyp:loc-max}, let $X\in\cat{T}$ and $c\in\cat{T}^c$ non-zero such that $X\otimes c=0$. Then $\bar X=0$ in~$\bat{A}$.
\qed
\end{Cor}

Another upshot is that if there is more than one choice for~$\cat{B}$ then the corresponding pure-injectives interact in the way one would expect of field objects.
\begin{Cor}
\label{cor:fieldy}%
Suppose we are in the situation of Hypothesis~\ref{hyp:loc-max}, with $\cat{B}_1$ and $\cat{B}_2$ distinct, $\cat{T}^c$-maximal Serre $\otimes$-ideals. Let $E_i=E(\cat{B}_i)$ be the corresponding pure-injectives (Construction~\ref{cons:E}). Then $E_1\otimes E_2 \cong 0$.
\end{Cor}
\begin{proof}
 Recall from Theorem~\ref{thm:E-tens} that we have equalities $\overrightarrow{\cat{B}_i} = \Ker(\hat{E}_i \otimes -)$ in~$\cat{A}$. Since the $\hat{E}_i$ are flat, the $\otimes$-ideal $\Ker(\hat E_1\otimes \hat E_2\otimes-)$ is Serre, and it contains both $\cat{B}_1$ and $\cat{B}_2$. But the subcategories $\cat{B}_i$ are maximal Serre $\otimes$-ideals by Corollary~\ref{cor:all-generate}. By the assumption that $\cat{B}_1$ and $\cat{B}_2$ are distinct we deduce that $\Ker(\hat{E}_1 \otimes \hat{E}_2 \otimes -)$ cannot be a proper Serre $\otimes$-ideal of~$\cat{A}\fp$. Thus it contains $\hat{\unit}$ forcing $E_1\otimes E_2 \cong 0$.
\end{proof}
\begin{Rem}
The corollary tells us that if $\cat{B}$ is a $\cat{T}^c$-maximal Serre $\otimes$-ideal with associated pure-injective $E$ then $\Ker(E\otimes -)$ seems quite large, at least relatively speaking; it contains the pure-injective associated to any other choice of $\cat{T}^c$-maximal Serre $\otimes$-ideal. It is then natural to wonder if $\Ker(E\otimes -)$ satisfies some sort of maximality property itself.
\end{Rem}
\begin{Cor}
\label{cor:bar-f_Y}%
Under Hypothesis~\ref{hyp:loc-max}, let $Y\subseteq\Spc(\cat{T}^c)$ be a Thomason subset (\eg\ a closed subset with quasi-compact complement) and assume $Y$ non-empty. Consider the idempotent triangle $\bbe_Y\to \unit\to \bbf_Y\to \Sigma \bbe_Y$ associated to~$Y$; see~\cite{BalmerFavi11}. Then $\overline{\bbf_Y}=0$ and $\overline{\bbe_Y}\cong\bar\unit$.
\end{Cor}

\begin{proof}
Since $Y\neq\varnothing$, there exists a non-zero $c\in\cat{T}^c$ with $\supp(c)\subseteq Y$, that is, $c\otimes \bbf_Y=0$. We conclude by Corollary~\ref{cor:bar-0} for $X=\bbf_Y$.
\end{proof}

\begin{Rem}
The last corollary shows that $\boneda\colon X\mapsto \bar X$ is only conservative on compacts, as in Proposition~\ref{prop:all-generate}\,(a), but not on all objects. It also shows that~$\boneda$ can send non-compact objects of~$\cat{T}$ to finitely presented ones in~$\bat{A}$.
\end{Rem}

\begin{Rem}
Proposition~\ref{prop:all-generate}\,(b) and Corollaries~\ref{cor:bar-0} and~\ref{cor:bar-f_Y} indicate that $\bat{A}$ is somewhat ``small''. For instance, if the open complement of the closed point~$Y=\{(0)\}\subset\Spc(\cat{T}^c)$ is quasi-compact (\eg\ if $\Spc(\cat{T}^c)$ is noetherian) then every $X\in\cat{T}$ has the same image $\bar X=\overline{\bbe_{0}\otimes X}$ as the object $\bbe_{0}\otimes X$ which belongs to the localizing subcategory $\cat{T}_{(0)}$ of~$\cat{T}$ generated by the minimal non-zero tt-ideal~$\cat{T}^c_{(0)}$ of~$\cat{T}^c$.
\end{Rem}

Another indication of the smallness of~$\bat{A}$ is the following:

\begin{Thm}
\label{thm:end-nil}%
Under Hypothesis~\ref{hyp:loc-max}, the endomorphism ring $\Endcat{C}(\bar\unit)$ of the $\otimes$-unit in~$\bat{A}:=\cat{A}/\overrightarrow{\cat{B}}$ (\ie in~$\bat{A}\fp$) is a commutative local ring such that every element of its maximal ideal is nilpotent. In particular, it has Krull dimension zero.
\end{Thm}

The following lemma allows us to convert this into a problem about Serre $\otimes$-ideals, which we are by now well equipped to handle.

\begin{Lem}
\label{lem:nil}%
Let $\cat{C}$ be an abelian $\otimes$-category, with $\otimes$ right exact. Let $f\colon X\to Y$ be a morphism in~$\cat{C}$ with~$Y$ flat. Then
\[
\Nil(f):=\SET{M\in\cat{C}}{\exists\, n>0,\textrm{ s.t.\ }M\otimes f\potimes{n}\colon M\otimes X\potimes{n}\to M\otimes Y\potimes{n}\textrm{ is zero}}
\]
is a Serre $\otimes$-ideal of~$\cat{C}$.
\end{Lem}

\begin{proof}
It is easy to check that $\Nil(f)$ is $\otimes$-ideal and stable by quotients and subobjects (the latter uses that $Y$ is flat).
Consider an exact sequence $M'\overset{g}\into M\overset{h}\onto M''$ with $M',M''\in\Nil(f)$. Replacing $f$ by $f\potimes{n}$ for $n$ large enough, we can as well assume that $M'\otimes f=0$ and $M''\otimes f=0$. It suffices to show that $M\otimes f\potimes{2}=0$. Use flatness of~$Y$ to obtain the exact rows of the following commutative (plain) diagram, in which we replace $M'\otimes f$ and $M''\otimes f$ by zero:
\[
\xymatrix{
& M'\otimes X \ar[r]^-{g\otimes 1} \ar[d]_-{0}
& M\otimes X \ar[r]^-{h\otimes 1} \ar[d]^-{1\otimes f} \ar@{-->}[ld]^-{k}
& M''\otimes X \ar[r] \ar[d]^-{0} \ar@{-->}[ld]^(.4){\ell}
& 0
\\
0 \ar[r]
& M'\otimes Y \ar[r]_-{g\otimes 1}
& M\otimes Y \ar[r]_-{h\otimes 1}
& M''\otimes Y \ar[r]
& 0
}
\]
The vanishing of the diagonal in the left-hand square gives the existence of a morphism~$\ell$ such that $1\otimes f=\ell\circ (h\otimes 1)$ and similarly, the right-hand square gives the existence of~$k$ such that $1\otimes f=(g\otimes 1)\circ k$. Then we can tensor the first relation by~$Y$ on the right and the second one by~$X$ ``in the middle'' (meaning on the right and then swap the last two factors) to get the following two commuting triangles:
\[
\xymatrix@C=2em{
&&& M\otimes X\otimes X \ar[d]^-{1\otimes 1\otimes f} \ar[lld]_{(23)(k\otimes 1)(23)\ }
\\
& M'\otimes X\otimes Y \ar[rr]_-{g\otimes 1\otimes 1}
&& M\otimes X\otimes Y \ar[rr]^-{h\otimes 1\otimes 1} \ar[d]^-{1\otimes f\otimes 1}
&& M''\otimes X\otimes Y \ar[lld]^-{\ell\otimes1}
\\
&&& M\otimes Y \otimes Y
}
\]
From $hg=0$, it follows that the vertical composite $1_M\otimes f\otimes f$ is zero, as claimed.
\end{proof}

\begin{proof}[Proof of Theorem~\ref{thm:end-nil}]
All the statements are immediate consequences of the fact that every non-invertible element of~$\End_{\bat{A}}(\bar\unit)$ is nilpotent. Let $f\colon\bar\unit\to \bar\unit$ be non-invertible, so that either $\Ker(f)$ or~$\coker(f)$ is non-zero. From the commutative diagram
\[
\xymatrix@C=3em{
0 \ar[r]
& \Ker(f) \ar[r]^-{i} \ar[d]_-{1\otimes f} \ar@/^1em/[rd]_-{=0}
& \bar\unit \ar[r]^-{f} \ar[d]^-{f}
& \bar\unit \ar[r]^-{p} \ar[d]_-{f}
& \coker(f) \ar[r] \ar[d]_-{1\otimes f}
& 0
\\
0 \ar[r]
& \Ker(f) \ar[r]^-{i}
& \bar\unit \ar[r]^-{f}
& \bar\unit \ar[r]^-{p}
& \coker(f) \ar[r]
& 0
}
\]
one easily gets that $1_{\Ker(f)}\otimes f=0$ and $1_{\coker(f)}\otimes f=0$. In other words, $\Ker(f),\coker(f)$ belong to the Serre $\otimes$-ideal~$\Nil(f)$ of Lemma~\ref{lem:nil}. As one of them is non-zero and they both belong to~$\bat{A}\fp$, we know by Corollary~\ref{cor:all-generate} that they will generate the whole category $\bat{A}\fp$ as a Serre $\otimes$-ideal. Consequently $\Nil(f)=\bat{A}\fp\ni\bar\unit$ which means that $f\potimes{n}=0$ for some~$n>0$. As $\bar\unit$ is the unit in the symmetric monoidal category $\bat{A}$, composition in $\End_{\bat{A}}(\bar\unit)$ and tensoring endomorphisms coincide, which shows every non-invertible element is nilpotent as promised.
\end{proof}

\begin{Rem}
The ring~$\End_{\bat{A}}(\bar\unit)$ is local and receives the local ring $\Endcat{T}(\unit)$ via a ring homomorphism, which is a \emph{local} ring homomorphism since $c\mapsto \bar c$ is conservative (Proposition~\ref{prop:all-generate}\,(a)). Consequently, the image of the unique prime in~$\End_{\bat{A}}(\bar\unit)$ is the maximal ideal of~$\Endcat{T}(\unit)$, \ie the closed point of its Zariski spectrum.
\end{Rem}

\begin{center}
*\ *\ *
\end{center}
We conclude this section with a discussion of uniqueness of the $\cat{T}^c$-maximal Serre $\otimes$-ideal~$\cat{B}\subset \cat{A}\fp$ (Definition~\ref{def:max}) under the assumption that $\cat{T}$ is local (Definition~\ref{def:local}) and the closed point of $\SpcT$ is visible.

\begin{Hyp}
\label{hyp:min}%
Suppose that $\cat{T}$ is local, so the zero $\otimes$-ideal $(0)$ is prime in~$\cat{T}^c$. In addition we assume that $\gm=(0)$ is \emph{visible} in the sense of~\cite{BalmerFavi11}, meaning that the punctured spectrum $\Spc \cat{T}^c \setminus \{\gm\}$ is a \emph{quasi-compact} open. (This is automatic if the space~$\SpcT$ is noetherian.) Then the closed subset consisting of just the point $\gm$ corresponds to a thick $\otimes$-ideal $\cat{T}^c_{\gm}$ whose non-zero objects have support precisely~$\{\gm\}$. This is the minimal non-zero tt-ideal of~$\cat{T}^c$. We shall denote by
\[
\cat{T}_{\gm} := \Loc (\cat{T}^c_{\gm})
\]
the localizing $\otimes$-ideal of~$\cat{T}$ generated by $\cat{T}^c_{\gm}$.
\end{Hyp}

\begin{Rem}
Consider the idempotent triangle in~$\cat{T}$ in the sense of~\cite{BalmerFavi11}
\[
\bbe_{\gm} \to \unit \to \bbf_{\gm} \to \Sigma \bbe_{\gm}
\]
corresponding to the closed Thomason~$\{\gm\}\subset \Spc(\cat{T}^c)$. Recall that this exact triangle is characterized by the properties that $\cat{T}_{\gm}=\bbe_{\gm}\otimes\cat{T}=\Ker(\bbf_{\gm}\otimes-)$ and $\cat{T}/\cat{T}_{\gm} \cong \cat{T}_{\gm}^{\perp}=\bbf_{\gm}\otimes\cat{T}=\Ker(\bbe_{\gm}\otimes-)$.
\end{Rem}
\begin{Lem}
\label{lem:E-minimal}%
Let $\cat{B}$ be a $\cat{T}^c$-maximal Serre $\otimes$-ideal and let $E=E(\cat{B})$ be the corresponding pure-injective object of~$\cat{T}$ (Construction~\ref{cons:E}). Then $E$ lies in $\cat{T}_{\gm}$.
\end{Lem}

\begin{proof}
We saw in Corollary~\ref{cor:bar-f_Y} that $\bar\bbf_{\gm}=0$. It follows from Corollary~\ref{cor:E-tens} applied to the object~$X=\bbf_{\gm}$ that $E\otimes \bbf_{\gm}=0$, meaning $E\in\Ker(\bbf_{\gm}\otimes-)=\cat{T}_{\gm}$.
\end{proof}
\begin{Thm}
\label{thm:unique}%
Suppose that $\cat{T}$ is local with localizing $\otimes$-ideal~$\cat{T}_{\gm}$ at the closed point as in Hypothesis~\ref{hyp:min}. Let $\cat{B}$ be a $\cat{T}^c$-maximal Serre $\otimes$-ideal and $E=E(\cat{B})$ the associated pure-injective as in Construction~\ref{cons:E}. If
\begin{displaymath}
\{X\in \cat{T}_{\gm} \; \vert \; E\otimes X \cong 0\}
\end{displaymath}
contains no non-zero pure-injective, for instance if it is trivial, then $\cat{B}$ is the unique $\cat{T}^c$-maximal Serre $\otimes$-ideal of~$\mT$.
\end{Thm}
\begin{proof}
Suppose $\cat{B}'$ were another, distinct, $\cat{T}^c$-maximal Serre $\otimes$-ideal with associated pure-injective $E'$. By Lemma~\ref{lem:E-minimal} we know that $E'$ lies in $\cat{T}_{\gm}$ and by Corollary~\ref{cor:fieldy} we know that $E\otimes E' = 0$. But this would contradict our hypothesis and so such a $\cat{B}'$ cannot exist.
\end{proof}

We now exhibit a situation that (likely due to our lack of general knowledge) seems to occur frequently and in which we can apply the theorem. Recall that $\cat{T}_{\gm}$ is said to be \emph{minimal} if it contains no proper non-trivial localizing $\otimes$-ideal. In other words every non-zero object of $\cat{T}_{\gm}$ generates it as a localizing $\otimes$-ideal.

\begin{Exa}\label{ex:class}
When the localizing $\otimes$-ideals of~$\cat{T}$ are classified by the subsets of~$\SpcT$, then it is clear that $\cat{T}_{\gm}$ is minimal, since the corresponding subset of~$\SpcT$ is the one-point closed subset~$\{\gm\}$.
\end{Exa}

\begin{Rem}
\label{rem:min}%
Minimality of~$\cat{T}_{\gm}$ implies that if $E_1,E_2\in\cat{T}_{\gm}$ satisfy $E_1\otimes E_2=0$ then $E_1=0$ or $E_2=0$. Indeed, if $E_2\neq0$, we see that $\cat{T}_{\gm}\cap \Ker(E_1\otimes-)$ is a non-zero (it contains~$E_2$) localizing $\otimes$-ideal contained in~$\cat{T}_{\gm}$, hence it must be equal to it by minimality. This gives $\bbe_{\gm}\in\cat{T}_{\gm}=\Ker(E_1\otimes-)$ and $E_1\simeq \bbe_{\gm}\otimes E_1\simeq 0$.
\end{Rem}

\begin{Cor}
\label{cor:unique}%
Suppose that $\cat{T}$ is local as in Hypothesis~\ref{hyp:min}, with minimal localizing $\otimes$-ideal~$\cat{T}_{\gm}$. Then there is a unique $\cat{T}^c$-maximal Serre $\otimes$-ideal $\cat{B}$ in~$\mT$.
\end{Cor}

\begin{proof}
By Proposition~\ref{prop:max} we know such a $\cat{B}$ exists. By Remark~\ref{rem:min} we know tensoring with the corresponding pure-injective kills no non-zero object of $\cat{T}_{\gm}$. Combining these two facts we can apply Theorem~\ref{thm:unique} to conclude that $\cat{B}$ is unique as claimed.
\end{proof}

As alluded to in Example~\ref{ex:class}, the corollary applies to a number of situations: in all of the cases where we have a classification of localizing $\otimes$-ideals it arises, more or less, by showing that $\cat{T}_{\gm}$ is minimal. Let us give a concrete example.

\begin{Exa}
\label{ex:min1}%
Let $R$ be a commutative noetherian local ring with residue field $k$ and consider $\cat{T} = \mathsf{D}(R)$. Then, using Neeman's classification \cite{Neeman92a} we have that
\begin{displaymath}
\cat{T}^c_{\gm} = \{X \in \mathsf{D}^\mathrm{perf}(R) \mid \oplus_i H^i(X) \text{ has finite length}\},
\end{displaymath}
$\cat{T}_{\gm} = \Loc(k)$ the localizing subcategory generated by the residue field, and this latter subcategory is minimal. Thus Corollary~\ref{cor:unique} applies and there is a unique $\cat{T}^c$-maximal Serre $\otimes$-ideal. We note that, even in Krull dimension $1$, computing the homological residue field $\bat{A}$, as in Theorem~\ref{thm:summary}, is non-trivial.
\end{Exa}

As one would hope tt-fields also admit unique homological residue fields.

\begin{Exa}
\label{ex:min2}%
Let $\cat{F}$ be a tt-field. Then $\Spc \cat{F}^c$ is a point by Proposition~\ref{prop:spcF} and one can moreover show that $\cat{F}_{\gm} = \cat{F}$ is minimal; it follows that $0$ is the unique $\cat{F}^c$-maximal Serre $\otimes$-ideal (see Theorem~\ref{thm:field}). In particular, $\bat{A} = \MF$ verifies the hypotheses of Theorem~\ref{thm:summary}.
\end{Exa}

\begin{Rem}\label{rem:manyfields}
All of this suggests that, in situations where the closed point $\gm\in\SpcT$ has quasi-compact complement but $\cat{T}_{\gm}$ is \emph{not} minimal, it could be instructive to examine the collection of all $\cat{T}^c$-maximal Serre ideals. The naive guess is that, given there is a unique homological residue field when $\cat{T}_{\gm}$ is minimal, the structure of the collection of $\cat{T}^c$-maximal Serre ideals may measure failure of minimality. One could, somewhat speculatively, hope to augment the information coming from the spectrum by this set of subcategories, i.e.\ to consider $\Spc \cat{T}^c$ together with some additional data at each visible point, and use this to try to classify localizing ideals.
\end{Rem}

%------------------------------------------------------------------------------
\goodbreak
\section{Examples and properties of tt-fields}
\label{se:field}%
\medbreak
%------------------------------------------------------------------------------

This section should help readers build an intuition of what tt-fields should be. We point out why some naive guesses are not appropriate solutions.

We start with some examples coming from modular representation theory, which seem quite different from `classical' fields but nonetheless should be admitted as tt-fields.
\begin{Prop}
\label{prop:C_p}%
Let $p$ be a prime, $C_p$ the cyclic group with~$p$ elements and $\kk$ a field of characteristic~$p$. Let $\cat{F}=\Stab(\kk C_p)$ be the stable module category of $\kk C_p$-modules modulo projectives. Then we have:
\begin{enumerate}[\rm(a)]
\item
Any object of~$\Stab(\kk C_p)$ is a coproduct of compact ones.
\smallbreak
\item
Any non-zero compact object of~$\Stab(\kk C_p)$ is \tff.
\smallbreak
\item
Any coproduct-preserving tt-functor $F\colon\cat{F} \to \cat{S}$ into a `big' tt-category (Hypothesis~\ref{hyp:big}) is faithful.
\smallbreak
\item
Every object of $\cat{F}$ is pure-injective i.e.\ $\cat{F}$ is pure-semisimple (cf.\ Theorem~\ref{thm:semisimple}). Moreover, every indecomposable object of $\cat{F}$ is endofinite.
\end{enumerate}
\end{Prop}

The assertions of this proposition remain true (with the same proof) when $\kk C_p$ is
replaced by the distribution algebra of the finite group scheme $\alpha_p$, i.e.\ when we look at representations of the finite group scheme $\alpha_p$. Here $\alpha_p$ is the subgroup scheme of $\mathbb{G}_a$ which is defined by $\alpha_p(R) = \{r\in R \mid r^p = 0\}$. One can also consider representations of $\mu_p$, the subgroup scheme of $\mathbb{G}_m$ consisting of $p$th roots of unity, and this also gives a tt-field, although of a different flavour. In the case of $\mu_p$ the category of representations is semisimple abelian and equivalent, as a tt-category, to $\mathbb{Z}/p\mathbb{Z}$-graded vector spaces with the trivial triangulation (cf.\ Example~\ref{exa:F_0xF_1}). We are grateful to Burt Totaro for the suggestion to consider $\mu_p$ and $\alpha_p$.

\begin{proof}
  Part~(a) is well-known and follows from the fact that the algebra
  $\kk C_p$ is of finite
  representation type, see~\cite{Auslander1974}
  or~\cite{RiTa1974}. Part~(b) follows easily from the description
  of indecomposable finite-dimensional $\kk C_p$-modules, of which
  there is exactly one,~$[i]$, which is of dimension~$i$, for each
  $1\le i\le p-1$. See details in Example~\ref{exa:C_p} below. The
  dimension of~$[i]$ is invertible in~$\kk$ and, being the trace of
  the identity, it factors as $\unit\to [i]\otimes [i]^\vee\to \unit$;
  this shows that $[i]$ is \tff. Part~(c) will follow from~(b) by a
  general argument given below in
  Corollary~\ref{cor:faithful}. Similarly, (a) implies $\cat{F}$
  pure-semisimple by a general result of~\cite{Krause00,Beligiannis00}
  (see Theorem~\ref{thm:semisimple} below). Endofiniteness then
  follows from finite-dimensionality over~$\kk$ of the homomorphism
  spaces in~$\cat{F}^c$.
\end{proof}

\begin{Exa}
\label{exa:C_p}%
With the notation of the proposition, if~$\sigma\in C_p$ is a generator and we let $t=\sigma-1$ then $\kk C_p\simeq \kk[t]/t^p$. The indecomposable non-projective modules are $[i]=\kk[t]/t^i$, for $i=1,\ldots,p-1$. One can show that for $i\le j$:
\begin{equation}
\label{eq:i@j}%
[i]\otimes [j]\simeq\left
\{\begin{array}{cl}{} [j-i+1]\oplus [j-i+3] \oplus \cdots\oplus [j+i-1]
\kern2em & \textrm{if } i+j\le p\\{} [j-i+1]\oplus [j-i+3]
\oplus \cdots\oplus [2p-i-j-1] & \textrm{if } i+j > p.
\end{array}\right.
\end{equation}
See~\cite[Cor.\,10.3]{CarlsonFriedlanderPevtsova08}. This formula indicates that the tensor can become rather involved even in what we want to think of as a tt-field.
\end{Exa}

We shall use a couple of times the following consequence of Brown-Neeman representability. See~\cite[Prop.\,2.15]{BalmerDellAmbrogioSanders16}.
\begin{Prop}
\label{prop:BN-rep}%
Let $F\colon\cat{T}\to \cat{S}$ be a coproduct-preserving tt-functor between `big' tt-categories (Hypothesis~\ref{hyp:big}). Then $F$ admits a right adjoint $U\colon\cat{S}\to \cat{T}$ satisfying the projection formula for every $X\in\cat{T}$ and $Y\in\cat{S}$:
\begin{equation}
\label{eq:proj-formula}%
X\otimes U(Y) \cong U(F(X)\otimes Y)\,.
\end{equation}
\end{Prop}

\begin{Cor}
\label{cor:faithful}%
Let $F\colon\cat{T}\to \cat{S}$ be a coproduct-preserving tt-functor between `big' tt-categories and suppose that every non-zero object of~$\cat{T}$ is \tff\ (\eg\ this holds if~$\cat{T}$ is a tt-field, Definition~\ref{def:field}). Then $F$ is faithful.
\end{Cor}

\begin{proof}
We use the notation of Proposition~\ref{prop:BN-rep}. From~\eqref{eq:proj-formula} for~$Y=\unit_{\cat{S}}$ comes $UF\simeq U(\unit_{\cat{S}})\otimes-$.
Since
\[
\Homcat{T}(\unit_{\cat{T}}, U(\unit_{\cat{S}})) \cong \Homcat{S}(F(\unit_{\cat{T}}), \unit_{\cat{S}}) \cong \Homcat{S}(\unit_{\cat{S}}, \unit_{\cat{S}}) \neq 0
\]
we see that $U(\unit_{\cat{S}})$ is a non-zero object of~$\cat{T}$. Hence $UF\simeq U(\unit_{\cat{S}})\otimes-$ is faithful by assumption. Hence~$F$ is faithful.
\end{proof}

We shall see in Remarks~\ref{rem:not-faithful-1} and~\ref{rem:not-faithful-2} that the hypotheses of Corollary~\ref{cor:faithful} are necessary. We also want to explain why being a pure-semisimple triangulated category is not sufficient to be a reasonable candidate for fieldness. Recall the (non-tensor) notion of pure-semisimplicity:
\begin{Thm}
\label{thm:semisimple}%
Let $\cat{F}$ be a compactly generated triangulated category. The following are equivalent:
\begin{enumerate}[\rm(i)]
\item
$\cat{F}$ is \emph{pure-semisimple}, \ie every object $X$ of~$\cat{F}$ is pure-injective (\ie $\hat X$ is injective in~$\MF$).
\smallbreak
\item
$\cat{F}$ is \emph{phantomless}, \ie $\yoneda\colon\cat{F}\to \MF$ is faithful (equivalently fully-faithful).
\smallbreak
\item
\label{it:ss-coprod}%
Every object of~$\cat{F}$ is a coproduct of compacts (that we can assume indecomposable with local endomorphism ring). In particular, $\cat{F}^c$ is Krull-Schmidt.
\end{enumerate}
\end{Thm}

\begin{proof}
See~Beligiannis~\cite[Thm.\,9.3]{Beligiannis00} which also gives a host of other equivalent formulations. Some of these results appear independently in~\cite[Thm.\,2.10]{Krause00}.
\end{proof}

\begin{Rem}
\label{rem:loc-fin}%
Another equivalent condition is that $\MF$ is \emph{Frobenius}, \ie
injectives and projectives coincide. When $\cat{F}^c$ is rigid, the
dual $(-)^\vee$ gives an equivalence $\cat{F}^c\cong(\cat{F}^c)\op$
and therefore $\Mod((\cat{F}^c)\op)\cong \MF$ is also Frobenius. It
then follows from~\cite[Prop.\,9.2]{Beligiannis00} that $\MF$ is
\emph{locally finite}, meaning that its finitely generated objects are
(finitely presented and) of finite length. This length function
on~$\cat{F}^c$ is an interesting invariant that shall be studied in
another work. The fact that $\MF$ is locally finite implies that
every indecomposable object of~$\cat{F}$
is \emph{endofinite}; see~\cite[\S\,3]{KrauseReichenbach01}.
\end{Rem}

Another illustration of the smallness of tt-fields is the following:
\begin{Prop}
\label{prop:dual}%
Let $\cat{F}$ be a `big' tt-category which is pure-semisimple (for instance a tt-field) and $\cat{A}=\MF$ its module category. Then there exists an exact functor $\ihomcat{A}(-,\hat\unit)\colon\cat{A}\op\to \cat{A}$ which restricts to an involution $(\cat{A}\fp)\op\isoto \cat{A}\fp$ on the finitely presented objects.
\end{Prop}

\begin{proof}
As $\hat\unit$ is injective, the functor $D=\ihomcat{A}(-,\hat\unit)\colon\cat{A}\op\to \cat{A}$ is exact by Lemma~\ref{lem:ihominj}. Let $M\in\cat{A}\fp$ and $x,y\in \cat{F}$ such that $M=\im(f)$ for some $f\colon x\to y$ in~$\cat{F}$. A direct verification shows that $D(\hat x)\cong\widehat{x^\vee}$ and exactness of~$D$ gives $D(M)\simeq\im(\widehat{f^\vee})$ and therefore $D^2(M)\cong\im(\widehat{f^{\vee\vee}})\cong\im(\hat{f})=M$.
\end{proof}

\begin{Rem}
\label{rem:ss-faithful}%
The assumption about $\cat{F}$ being pure-semisimple removes any ambiguity about the meaning of `every non-zero object is $\otimes$-faithful'. Suppose the weakest form, namely that the functor $x\otimes-\colon\cat{F}^c\to \cat{F}^c$ is faithful for every non-zero compact $x\in\cat{F}^c$. Then it immediately follows from Theorem~\ref{thm:semisimple}\,\eqref{it:ss-coprod} that $x\otimes-\colon\cat{F}\to \cat{F}$ is faithful as well and then any non-zero $X\otimes-\colon\cat{F}\to\cat{F}$ is faithful.
\end{Rem}

\begin{Exa}
\label{exa:C_p^n}%
Let $C_{p^n}$ denote the cyclic group with $p^n$ elements and let $\kk$ be a field of characteristic $p$. The stable category $\sMod{\kk C_{p^n}}$ is pure-semisimple. This follows, for instance, from the fact that $\kk C_{p^n}$ has finite representation type (see~\cite[\S\,12]{Beligiannis00}). Moreover, the compact part $\smodname \kk C_{p^n}$ is local and its spectrum is a single point. However, $\sMod \kk C_{p^n}$ should morally not be a tt-field for $n\ge 2$. Indeed, restriction along the inclusion $C_p \to C_{p^n}$ gives a functor
\begin{displaymath}
\sMod \kk C_{p^n} \to \sMod \kk C_p
\end{displaymath}
which should be regarded as a residue field. In other words, $\Stab \kk C_{p^n}$ can be made `smaller'. Indeed, the tt-category $\sMod \kk C_{p^n}$ is not a tt-field in the sense of Definition~\ref{def:field} as the $p^{n-1}$-dimensional module $\kk(C_{p^n}/C_{p})$ is not $\otimes$-faithful.
\end{Exa}

\begin{Exa}
\label{exa:Q_8}%
Consider the group of quaternions~$Q_8$. As its center~$C_2$ is the maximal elementary abelian subgroup, $\Spc(\smod\kk Q_8)$ is a point. In this case, the residue field functor is probably given by restriction from~$Q_8$ to~$C_2$. It is another amusing case of something like an artinian local tt-category whose residue field is still (probably) an \'etale extension in the sense of~\cite{BalmerDellAmbrogioSanders15}.
\end{Exa}

\begin{Rem}
\label{rem:not-faithful-1}%
The faithfulness of any coproduct-preserving tt-functor $F\colon\cat{F}\to \cat{S}$ out of a tt-field~$\cat{F}$ (Corollary~\ref{cor:faithful}) cannot hold if~$\cat{F}$ is merely pure-semisimple. Indeed, as in Example~\ref{exa:C_p^n}, let $G=C_{p^n}$ with $n\ge 2$ and $C_p<G$ the maximal elementary abelian. Then $\Res^G_{C_{p}}\colon\Stab(\kk G)\to \Stab(\kk C_p)$ is a non-faithful tt-functor and $\Stab(\kk G)$ is pure-semisimple.
\end{Rem}

One might think that if an object $X$ of a `big' tt-category~$\cat{T}$ has the property that the kernel $X\otimes(-)$ contains no non-zero compact object then $X$ should survive in the residue field.  The following example illustrates that this is not the case (cf.\ Corollary~\ref{cor:bar-0}).

\begin{Exa}
Let $\cat{T} = \SH_{(p)}$ denote the $p$-local stable homotopy category and $IS$ the Brown-Comenetz dual of the sphere spectrum. This object is characterized by the existence of natural isomorphisms $\Homcat{T}(?,IS)\cong \Hom_{\bbZ}(\pi_0(?),\bbQ/\bbZ)$. The functor $IS\wedge -$ has no kernel on finite spectra, so one might suspect it should survive in some residue field. However, $IS\wedge IS \cong 0$ which suggests that in fact $IS$ must become zero in any residue field. Details and references for the facts in this example can be found in \cite[Section~7]{HoveyPalmieri99}.
\end{Exa}

\begin{center}
*\ *\ *
\end{center}

We now discuss further properties of tt-fields in the sense of our Definition~\ref{def:field}.

\begin{Prop}
\label{prop:spcF}%
Let $\cat{F}$ be a tt-category such that every object of~$\cat{F}^c$ is $\otimes$-faithful (for instance, a tt-field). Then $\SpcF$ is a point.
\end{Prop}

\begin{proof}
We need to show that any non-zero object~$x\in\cat{F}^c$ generates the whole category as a $\otimes$-ideal i.e.\ $\ideal{x}=\cat{F}^c$. We know that $x\otimes\coev _x$ is a split monomorphism, where $\coev_x\colon\unit\to x^\vee\otimes x$ is the unit of the $x\otimes-\adj x^\vee\otimes-$ adjunction. It follows from faithfulness of~$x\otimes-$ that $\coev_x$ is a split monomorphism, hence $\unit\in\ideal{x}$.
\end{proof}

\begin{Rem}
The converse to Proposition~\ref{prop:spcF} does not hold, as can be seen on $\cat{T}=\Der(R)$ for $R$ artinian local, not a field, say $R=\kk[t]/t^2$ for $\kk$ a field.
\end{Rem}

\begin{Thm}
\label{thm:field}%
Let $\cat{F}$ be a tt-field in the sense of Definition~\ref{def:field}. Then we have:
\begin{enumerate}[\rm(a)]
\smallbreak
\item
\label{it:field-Serre}%
The only proper Serre $\otimes$-ideal $\cat{B}\subset\mF$ is zero.
\smallbreak
\item
\label{it:field-hom-faithful}%
For every non-zero $X$ in~$\cat{F}$, the functors $\ihom(X,-)$ and $\ihom(-,X)$ are faithful.
\smallbreak
\item
\label{it:field-cohomology}%
Let $X\in\cat{F}$ be a non-zero object and $f\colon Y\to Z$ a morphism in~$\cat{F}$ such that $\Homcat{F}(c\otimes f, X)=0$ for all $c\in\cat{F}^c$. Then $f=0$.
\end{enumerate}
\end{Thm}

\begin{proof}
For \eqref{it:field-Serre}, let $E=E(\cat{B})\in\cat{F}$ be the pure-injective associated to~$\cat{B}$ as in Construction~\ref{cons:E}. As every $M\in\mF$ is the image of $f\colon x\to y$ in~$\cat{F}^c$, Corollary~\ref{cor:E-residual} and \tff ness of~$E$ imply that $\cat{B}=0$ as wanted.

For~\eqref{it:field-hom-faithful}, we have from Theorem~\ref{thm:semisimple}\,\eqref{it:ss-coprod} that $X\simeq \coprod_{i\in I}x_i$ with $x_i\in \cat{F}^c$ compact. Hence $\ihom(X,-)\simeq \prod_{i\in I}\ihom(x_i,-)$ and $\ihom(x_i,-)\cong x_i^\vee\otimes-$ is faithful for $x_i$ non-zero. On the other hand, since each $x_i$ is a summand of~$X$, it follows that $\ihom(-,x_i)$ is a summand of~$\ihom(-,X)$ and it suffices to prove that $\ihom(-,x_i)$ is faithful. We have $\ihom(-,x_i)\cong \ihom(-,\unit)\otimes x_i$. Indeed, this holds when the source is compact, which is sufficient up to pulling out coproducts as products since the category is pure-semisimple. We are left to show that $(-)^\vee=\ihom(-,\unit)$ is faithful on the whole of~$\cat{F}$. This is immediate from Theorem~\ref{thm:semisimple}\,\eqref{it:ss-coprod} again: Let $f\colon Y\to Z$ be a morphism in~$\cat{F}$ and $Y\simeq \coprod_{j\in J} y_j$ and $Z\simeq\coprod_{k\in K} z_k$ with $y_j,\,z_k\in\cat{F}^c$. Then $f$ is characterized by $f_{kj}\colon y_j\to z_k$ (using compactness of the~$y_j$). Let $j\in J$ and $k\in K$ such that $f_{kj}\colon y_j\into y \oto{f} z\onto z_k$ is non-zero. Then, as $(-)^\vee\colon(\cat{F}^c)\op\to \cat{F}^c$ is an involution, we have $0\neq f_{kj}^\vee\colon z_k^\vee\to z^\vee\oto{f^\vee}y^\vee\to y_j^\vee$, showing that $f^\vee\neq0$.

For~\eqref{it:field-cohomology}, it now suffices to show that the morphism $\ihom(f,X)$ is zero in~$\cat{F}$. As $\cat{F}$ is phantom-free (Theorem~\ref{thm:semisimple}), it suffices to show that this morphism is a phantom, \ie that it maps to zero under every $\Homcat{F}(c,-)$ for $c\in\cat{F}^c$. The result now follows from adjunction: $\Homcat{F}(c,\ihom(f,X))\cong\Homcat{F}(c\otimes f,X)=0$.
\end{proof}

\begin{Exa}
It follows easily from Theorem~\ref{thm:field}\,\eqref{it:field-Serre} that for every non-zero object $X\in\cat{F}$ and $M\in\mF$, we have $\hat X\otimes M=0$ only if $M=0$. In fact the argument given in the proof, replacing $E$ by $X$, shows this.
\end{Exa}

\begin{Exa}
\label{exa:F_0xF_1}%
One cannot conclude from Theorem~\ref{thm:field} that every proper Serre subcategory of~$\mF$ is zero, without the $\otimes$-ideal assumption. In fact, $\cat{F}$ can even additively decompose itself, as the following simple example shows. Let $\cat{F}_0$ be a tt-field in the sense of Definition~\ref{def:field}, for instance the derived category of a field~$k$. Let $\cat{F}_1:=\cat{F}_0$ another copy of the same triangulated category and $\cat{F}=\cat{F}_0\times \cat{F}_1$ with component-wise morphisms. We define the tensor product in a $\bbZ/2$-graded way:
\[
(x_0,x_1)\otimes (y_0,y_1):=\big((x_0\otimes y_0)\oplus (x_1\otimes y_1)\,,\, (x_0\otimes y_1)\oplus (x_1\otimes y_0)\big)\,.
\]
This makes~$\cat{F}$ into a tt-category; for instance $\unit_{\cat{F}}=(\unit,0)$. It is easy to verify that $\cat{F}$ remains a tt-field in the sense of Definition~\ref{def:field}. Moreover, $(0,\unit)$ is invertible of order two, so every object of $\cat{F}$ is a direct sum of invertible objects.

On the other hand, $\cat{F}$ displays some behaviour that is, at least at first glance, not desirable from a provincial point of view on fields. There is a natural collapsing functor
\begin{displaymath}
\pi\colon \cat{F} \to \cat{F}_0 \quad \text{defined by} \quad (x_0,x_1) \longmapsto x_0\oplus x_1.
\end{displaymath}
One easily checks that $\pi$ is strong monoidal and faithful; our field $\cat{F}$ faithfully (but not fully) embeds via a tensor functor into a `smaller' field $\cat{F}_0$. This reflects the situation at the level of abelian categories: the category of $\bbZ/2\bbZ$-graded $k$-vector spaces (of which $\cat{F}$ is the derived category) exhibits the same embedding into ungraded $k$-vector spaces.
\end{Exa}

\begin{Rem}
\label{rem:not-faithful-2}%
The above example also shows that one cannot simply define a field by requesting that \emph{every} triangulated functor out of~$\cat{F}$ be faithful. Indeed, the projection on the first factor~$\cat{F}\onto \cat{F}_0$ is not faithful. Note that it is not a tt-functor (compare Corollary~\ref{cor:faithful}). Example~\ref{exa:F_0xF_1} also shows that Corollary~\ref{cor:faithful} cannot hold for the adjoints of tt-functors. Indeed, the projection $\cat{F}=\cat{F}_0\times\cat{F}_1\onto \cat{F}_0$ is (two-sided) adjoint to the coproduct-preserving tt-functor $\cat{F}_0\hook\cat{F}$ given by inclusion.
\end{Rem}

Here is a very short characterization of being a tt-field (Definition~\ref{def:field}).
\begin{Thm}
\label{thm:field-ihom}%
A `big' tt-category $\cat{F}$ is a tt-field if and only if $\cat{F}\neq0$ and the functor $\ihom(-,X)\colon\cat{F}\op\to \cat{F}$ if faithful for every non-zero object~$X\in\cat{F}$.
\end{Thm}

\begin{proof}
We have already seen in Theorem~\ref{thm:field}\,\eqref{it:field-hom-faithful} that the condition is necessary. Conversely, choose a non-zero pure-injective $X\in\cat{F}$. By Proposition~\ref{prop:phantom}\,\eqref{it:phantom-b}, the functor $\ihom(-,X)$ vanishes on phantom maps; but we are also assuming that this functor is faithful. It follows that $\cat{F}$ is phantomless and we conclude by Theorem~\ref{thm:semisimple} that every object is a coproduct of compacts. It follows that for every $x\in\cat{T}^c$, the functor $\ihom(-,x)\simeq (-)^\vee\otimes x$ is faithful, and in particular on compacts $-\otimes x$ becomes faithful, which we know suffices by Remark~\ref{rem:ss-faithful}.
\end{proof}

We note for future use that the tt-subcategories of tt-fields are fields themselves:

\begin{Prop}
\label{prop:full-subfield}%
Let $\cat{F}$ be a tt-field and $F\colon\cat{G}\hook\cat{F}$ be a \emph{fully-faithful} coproduct-preserving tt-functor, where $\cat{G}$ is rigidly-compactly generated. Then $\cat{G}$ is a tt-field.
\end{Prop}

\begin{proof}
As $F$ is fully-faithful and~$\cat{G}$ is idempotent-complete, for every $X\in\cat{G}$, every summand of $F(X)$ in~$\cat{F}$ must come from~$\cat{G}$ and it follows that if $F(X)\simeq\coprod_{i\in I}y_i$ in~$\cat{F}$ then each $y_i\simeq F(x_i)$ for $x_i\in\cat{G}$ and $X\simeq \coprod_{i\in I}x_i$ in~$\cat{G}$. As $F$ is monoidal, it preserves rigids, hence compacts. Conversely, if $F(x)$ is compact, it is clear from $F$ being fully-faithful and coproduct-preserving that $\Homcat{G}(x,-)\cong\Homcat{F}(F(x),F(-))$ will commute with coproducts; hence such $x\in\cat{G}$ is compact. We have shown that every $X\in\cat{G}$ is a coproduct of compacts. It is clear from $F$ being faithful and a $\otimes$-functor that every non-zero $X\in\cat{G}$ is $\otimes$-faithful.
\end{proof}

%------------------------------------------------------------------------------
\goodbreak
\section{Abelian residues of tt-residue fields}
\label{se:fairyland}%
\medbreak
%------------------------------------------------------------------------------

In this section we explore some consequences of the hypothetical existence of a tt-residue field. We emphasize those phenomena which match the results we proved on module categories in Sections~\ref{se:MT}-\ref{se:max}.

Let $\cat{T}$ be a `big' tt-category (Hypothesis~\ref{hyp:big}) which is~\emph{local} (Definition~\ref{def:local}). Let us then clarify what we tentatively mean by the existence of a tt-residue field.

\begin{Hyp}
\label{hyp:res}%
Assume we are given a coproduct preserving tt-functor
\[
F\colon \cat{T}\too \cat{F}
\]
satisfying the following three conditions:
\begin{enumerate}[\rm(1)]
\item
\label{it:res1}%
The tt-category $\cat{F}$ is a tt-field in the sense of Definition~\ref{def:field}. See also Section~\ref{se:field}.
\smallbreak
\item
\label{it:res2}
The functor $F$ is conservative on the compact-rigid objects, \ie the preimage of zero $\SET{x\in\cat{T}^c}{F(x)=0}$ is the unique closed point~$(0)$ of~$\SpcT$.
\smallbreak
\item
\label{it:res3}
Every object $Y$ of~$\cat{F}$ is a summand of the image $F(X)$ of some~$X\in\cat{T}$.
\end{enumerate}
\end{Hyp}

\begin{Rem}\label{rem:compactpres}
Any strong monoidal functor preserves rigid objects\textemdash{}having a dual is defined in terms of identity maps, composition, and tensor products, all of which are preserved by a strong monoidal functor. Thus it is automatic from the hypotheses (those above together with the always standing Hypothesis~\ref{hyp:big}) that $F$ sends compact objects to compact objects.
\end{Rem}

\begin{Rem}
We have already motivated Condition~\eqref{it:res2} in Section~\ref{se:intro}. It means that the induced map $\Spc(F)\colon\SpcF\to \SpcT$ sends the unique point of~$\SpcF$ (Proposition~\ref{prop:spcF}) to the unique closed point of~$\SpcT$, as one expects with residue fields. This property can be seen as a tt-generalization of Nakayama for local rings.

Condition~\eqref{it:res3} is a tentative and imperfect way to express that $\cat{F}$ is some sort of `quotient' of~$\cat{T}$, \ie that $\cat{F}$ is not too far from the image subcategory~$F(\cat{T})$. This is inadequate for several reasons, in particular because it does not prevent $F\colon\cat{T}\to \cat{F}$ from `overshooting the mark'. Indeed, even in commutative algebra, such an $\cat{F}$ could be a field extension of the actual residue field.

However, Condition~\eqref{it:res3} should be easy to verify in practice since $\cat{F}$ should be very small. For instance, it clearly holds if every object of~$\cat{F}$ is a coproduct of suspensions of~$\unit_{\cat F}=F(\unit_{\cat T})$, or a summand thereof, as in the topologists' definition.

An interesting open problem would be to replace Condition~\eqref{it:res3} by a more restrictive one which would capture the idea that~$\cat{F}$ is the `smallest' field into which~$\cat{T}$ maps, while still satisfying Condition~\eqref{it:res2}. This minimality could be stated by saying that any factorization of $F$ as $\cat{T}\to \cat{F}'\to \cat{F}$ via another tt-field~$\cat{F}'$ must be trivial: $\cat{F}'\isoto\cat{F}$. In view of Proposition~\ref{prop:full-subfield}, such `minimality' of~$\cat{F}$ easily implies that $\cat{F}$ is generated by~$F(\cat{T})$ as a localizing subcategory. On the other hand, it is not clear how much fullness of~$F$ could be obtained from minimality, nor whether Condition~\eqref{it:res3} would follow. Finally, it would be interesting to prove that any tt-functor to a field factors via a minimal one, but this also remains elusive. For these reasons, we do not include minimality among our hypotheses.
\end{Rem}

\begin{Not}
We denote by $U\colon\cat{F}\to \cat{T}$ the right adjoint to~$F$, which exists by Proposition~\ref{prop:BN-rep} and satisfies the projection formula~\eqref{eq:proj-formula}.
\end{Not}

\begin{Rem}
\label{rem:faithful-U}%
The rather mild Condition~\eqref{it:res3} already forces faithfulness of the right adjoint $U\colon\cat{F}\to \cat{T}$, as is well-known. Indeed, Condition~\eqref{it:res3} implies that the counit $\epsilon_Y\colon FU(Y)\to Y$ is a split epimorphism and thus every morphism $f\colon F(X)\to Y$ such that $U(f)=0$ must satisfy $f\circ\epsilon_{F(X)}=\epsilon_Y\circ FU(f)=0$, hence $f=0$.
\end{Rem}

Another immediate consequence of Condition~\eqref{it:res3} is the following:

\begin{Prop}
\label{prop:E-free}%
Under Hypothesis~\ref{hyp:res}, the adjunction~$F\adj U$ induces an equivalence
\[
(UF\,\textrm{\rm-Free}_{\cat{T}})^\natural \isotoo \cat{F}
\]
between the idempotent-completion of the Kleisli category of free modules over the monad~$UF$ on the category~$\cat{T}$ and the residue field~$\cat{F}$, in such a way that our functors $F\colon\cat{T}\to \cat{F}$ and $U\colon\cat{F}\to \cat{T}$ become respectively the free $UF$-module functor $\cat{T}\to (UF\,\textrm{\rm-Free}_{\cat{T}})^\natural$ and the underlying-object functor $(UF\,\textrm{\rm-Free}_{\cat{T}})^\natural \too\cat{T}$.
\end{Prop}

\begin{proof}
By general category theory~\cite[\S\,VI.5]{MacLane98}, the Kleisli comparison functor $K\colon UF\,\textrm{\rm-Free}_{\cat{T}} \to \cat{F}$ associated to the $(F,U)$-adjunction is fully-faithful. Our Condition~\eqref{it:res3} immediately implies that this functor~$K$ is surjective up to direct summand, hence an equivalence on idempotent-completions.
\end{proof}

\begin{Rem}
Given an exact monad $\bbM$, like the above $\bbM=UF$, on a triangulated category~$\cat{T}$, there is no guarantee that the Kleisli category or its idempotent-completion $(\bbM\textrm{\rm-Free}_{\cat{T}})^\natural$ be triangulated. In the above result, we heavily use that the monad is already realized by an exact adjunction of triangulated categories.
\end{Rem}

\begin{Not}
Let $\bbE=U(\unit_{\cat{F}})$ the image of the $\otimes$-unit of~$\cat{F}$ in~$\cat{T}$.
\end{Not}

\begin{Cor}
\label{cor:E}%
Under Hypotheses~\ref{hyp:res}, the object~$\bbE=U(\unit_{\cat{F}})$ of~$\cat{T}$ is equipped with the structure of a commutative ring object in~$\cat{T}$, coming from the fact that $U$ is lax-monoidal. We have an isomorphism of monads $UF(-)\cong \bbE\otimes-$ where the latter is a monad via the above ring structure. We have an equivalence of categories
\[
(\bbE\textrm{\rm-Free}_{\cat{T}})^\natural \isoto \cat{F}
\]
between the idempotent-completion of the Kleisli category of free $\bbE$-modules in~$\cat{T}$ and the tt-field~$\cat{F}$, in such a way that the functors $F\colon\cat{T}\to \cat{F}$ and $U\colon\cat{F}\to \cat{T}$ become the free $\bbE$-module and underlying-object functors respectively.
\end{Cor}

\begin{proof}
The right-adjoint~$U$ to the monoidal functor~$F$ is lax-monoidal, hence preserves (commutative) ring objects. This explains the ring structure on~$\bbE=U(\unit)$. The isomorphism of monads is a general consequence of the projection formula; see~\cite[Lemma~2.8]{BalmerDellAmbrogioSanders15}. The statement now follows from Proposition~\ref{prop:E-free}.
\end{proof}

We now want to discuss the module-category analogue of the above.

\begin{Cons}
Our basic framework is outlined in the following diagram
\begin{equation}\label{eq:setup}
\vcenter{\xymatrix@C=4em{
\cat{T} \ar[r]^-{\yoneda} \ar[d]<-0.5ex>_-F \ar@{<-}[d]<0.5ex>^-U
& \MT =:\cat{A} \kern-2em \ar[d]<-0.5ex>_-{\hat{F}} \ar@{<-}[d]<0.5ex>^-{\hat{U}} \\
\cat{F} \ \ar@{^(->}[r]^-{\yoneda} & \MF
}}
\end{equation}
where the two vertical pairs of functors are adjoint. The functor $\hat{F}\colon\MT\to \MF$ is the exact colimit-preserving functor that $F$ induces by the universal property (Remark~\ref{rem:hatG}), that is, the left Kan extension $(F^c)_!$ along the restriction $F^c$ of $F$ to compacts. This functor~$\hat{F}$ is strong monoidal. Its right adjoint $\hat{U}\colon\MF\to \MT$ is the restriction $(F^c)^*$ along $F^c\colon \cat{T}^c \to \cat{F}^c$ or, equivalently, the exact colimit-preserving functor induced by the universal property of the module category but now applied to~$U$.
Exactness of the functor $\hat{F}$ implies in particular that $\hat{U}$ preserves injectives or, equivalently, that $U$ preserves pure-injectives.
\end{Cons}

\begin{Cor}
\label{cor:E-inj}%
The ring-object~$\bbE=U(\unit_{\cat{F}})$ is pure-injective in~$\cat{T}$.
\end{Cor}

\begin{proof}
Every object of~$\cat{F}$ is pure-injective and $U$ preserves pure-injectives.
\end{proof}

The following lemma is immediate from the above discussion of the functor $\hat{U}$. However, we state it and indicate the proof for psychological reasons.
\begin{Lem}
Given $c\in \cat{F}^c$ there is a natural isomorphism $\hat{U}(\hat{c}) \cong \widehat{U(c)}$.
\end{Lem}

\begin{proof}
There is a chain of isomorphisms $\hat{U}(\hat{c}) = \cat{F}(F-,c) \cong \cat{T}(-, Uc) = \widehat{U(c)}$, the first by the definition of $\hat{U}$ as the restriction along $F^c$, the second by adjunction, and the third again by definition.
\end{proof}

\begin{Lem}\label{lem:U-faithful}
The functor $\hat{U}\colon\MF\to \MT$ is faithful.
\end{Lem}

\begin{proof}
Proving $\hat{U}$ is faithful is equivalent to showing that the counit $\epsilon\colon\hat{F}\hat{U}\to \Id$ of the adjunction between $\hat{F}$ and $\hat{U}$ is an epimorphism. For finitely presented projectives this is clear from the fact that $U$ is faithful: if $c\in \cat{F}^c$ we have $\hat{F}\hat{U}(\hat c)\cong \widehat{FU(c)}$ and under this identification $\epsilon$ is $\hat{\epsilon}$. Given this, we can use naturality of~$\hat\epsilon$, the fact that $\hat{F}$ and $\hat{U}$ are coproduct-preserving, and the fact that every object $M\in\MT$ is a quotient of a coproduct of finitely-presented projectives to deduce that every $\hat\epsilon_{M}$ is an epimorphism.
\end{proof}

\begin{Prop}
\label{prop:hatUmonadic}%
The object $\hat{\bbE}=\hat{U}(\hat\unit_{\cat{F}})$ of $\cat{A}=\MT$ is a commutative ring-object, which is injective. Moreover, $\hat{U}$ is monadic and identifies $\MF$ with the Eilenberg-Moore category $\hat{\bbE}\textrm{-}\Mod_{\cat{A}}$ of modules in $\cat{A}$ over this ring-object.
\end{Prop}

\begin{proof}
The functor $\hat{U}$ is faithful by Lemma~\ref{lem:U-faithful} and so we can apply Proposition~\ref{prop:monadic} to deduce that it is monadic. All that remains is to identify the monad $\hat{U}\hat{F}$ with the monad associated to $\hat{U}(\hat\unit_{\cat{F}})\cong\hat\bbE$. This is a consequence of Corollary~\ref{cor:E}.
\end{proof}

\begin{Rem}
It follows, more or less, from Remark~\ref{rem:loc-fin} that the object $\bbE$ is not just pure-injective, but actually endofinite. Indeed, by said remark the object $\unit_{\cat{F}}$ is endofinite, by virtue of the pure-semisimplicity of $\cat{F}$. To conclude that $\bbE$ is also endofinite it is enough to note that $U$ preserves endofiniteness. This follows from the fact that $U$ preserves products (as a right adjoint) and coproducts (its left adjoint preserves compacts cf.\ Remark~\ref{rem:compactpres}), see for instance \cite[Corollary~3.8]{KrauseReichenbach01}.
\end{Rem}

\begin{center}
*\ *\ *
\end{center}

We now wish to construct, in the spirit of Sections~\ref{se:MT} and \ref{se:E}, a quotient of $\cat{A}=\MT$ which is closer to~$\MF$.

\begin{Prop}
\label{prop:KerF}%
Under Hypothesis~\ref{hyp:res}, consider the localizing $\otimes$-ideal $\Ker(\hat{F})$ of~$\cat{A}=\MT$ and $\cat{B}=\Ker(\hat{F})\cap\cat{A}\fp$, its Serre $\otimes$-ideal subcategory of finitely presented objects. Then $\Ker(\hat{F})$ is locally finitely presented, that is,
\[
\Ker(\hat{F})=\overrightarrow{\cat{B}}.
\]
\end{Prop}

\begin{proof}
This follows from the fact that $\hat{F}$ preserves finitely presented objects and a general Grothendieck category argument. See Proposition~\ref{prop:fp-kernel}.
\end{proof}

\begin{Cons}
We can extend our diagram (\ref{eq:setup}) by factoring $\hat{F}$ as follows
\begin{equation}\label{eq:setup2}
\vcenter{\xymatrix@R=2em{
\cat{T} \ar[r]^-{\yoneda} \ar[dd]<-0.5ex>_-F \ar@{<-}[dd]<0.5ex>^-U
& \cat{A}=\MT \ar[dd]<-0.5ex>_-{\hat{F}} \ar@{<-}[dd]<0.5ex>^-{\hat{U}} \ar@{<-}[dr]<0.5ex>^-R \ar[dr]<-0.5ex>_-Q &
\\
&& \MT/\Ker(\hat F) =:\bat{A} \kern-2em \ar[dl]<-0.5ex>_-{\bar{F}} \ar@{<-}[dl]<0.5ex>^-{\bar{U}}
\\
\cat{F} \ar[r]^-{\yoneda}
& \MF &
}}
\end{equation}
Here $Q$ is the Gabriel quotient and~$R$ its right adjoint as in Proposition~\ref{prop:A/B}. The functor~$\bar F$ is characterized by~$\bar F\circ Q=\hat F$ and is therefore exact. By uniqueness of right adjoints, we must have $\hat U\cong R\circ \bar{U}$, and therefore, applying $Q(-)$, we get
\[
\bar U \cong Q\circ \hat{U}\,.
\]
\end{Cons}

\begin{Lem}\label{lem:barUmonadic}
In the above diagram, the functor $\bar{F}\colon\bat{A}\to \MF$ is faithful and strong monoidal, and $\bar{U}$ is faithful. In particular, $\bar{U}$ is monadic and we can identify $\MF$ with the Eilenberg-Moore category $\bar{\bbE}\textrm{-}\Mod_{\bat{A}}$ of $\bar{\bbE}$-modules in $\bat{A}$ over the ring-object $\bar{\bbE}=\bar{U}(\hat\unit_{\cat{F}})\cong Q\yoneda(\bbE)$.
\end{Lem}
\begin{proof}
It is clear that the functor $\bar{F}$ is faithful since it is conservative and exact (between abelian categories) and it is strong monoidal as is a factorization of the strong monoidal functor $\hat{F}$ through the monoidal localization corresponding to its kernel. We have seen in Lemma~\ref{lem:U-faithful} that $\hat{U}$ is faithful and so it follows from the isomorphism $\hat{U} \cong R\bar{U}$ that $\bar{U}$ is faithful. By Beck's monadicity theorem (Proposition~\ref{prop:monadic}) this shows $\bar{U}$ is monadic. Hence $\MF$ is the category of modules in~$\bat{A}$ over the monad $\bar{U}\bar{F}$. Now, $\bar{U}\bar{F}Q\cong Q\hat{U}\hat{F}\cong Q\widehat{UF}\cong Q(\hat{\bbE}\otimes-)\cong \bar{\bbE}\otimes Q$ and therefore the monad $\bar{U}\bar{F}\cong \bar{\bbE}\otimes-$ is again `monoidal' given by the ring $\bar{\bbE}$.
\end{proof}

\begin{Prop}
\label{prop:B-max}%
Under Hypothesis~\ref{hyp:res}, the Serre $\otimes$-ideal $\cat{B}=\Ker(\hat{F})\fp$ of~$\cat{A}\fp$ is $\cat{T}^c$-maximal (Definition~\ref{def:max}), \ie it is maximal among the Serre $\otimes$-ideals of~$\mT$ which do not contain any non-zero $\hat c$ for $c\in\cat{T}^c$.
\end{Prop}

\begin{proof}
Since $\hat{F}$ is monoidal its kernel $\Ker(\hat{F})$ is a $\otimes$-ideal. Thus $\cat{B}=(\Ker\hat{F})\fp$ is a Serre $\otimes$-ideal in $\cat{A}\fp$. We have assumed in Condition~\eqref{it:res2} that $F$ has no non-zero compact objects in its kernel so $\cat{B}$ does not meet $\SET{\hat c}{c\in\cat{T}^c,\ c\neq0}$ as wanted. Thus there exists a $\cat{T}^c$-maximal Serre $\otimes$-ideal $\cat{B}'$ containing~$\cat{B}$ and we claim it is equal to~$\cat{B}$. Suppose ab absurdo that there exists $M\in\cat{B}'\setminus\cat{B}$. This implies that $\hat F(M)\neq0$ in $\mF$. By Theorem~\ref{thm:field}\,\eqref{it:field-Serre}, we know that this non-zero object must generate the whole of~$\mF$ as a Serre $\otimes$-ideal. By Lemma~\ref{lem:id-Tc}, $\mF$ is generated as a Serre (non-ideal) subcategory by~$\SET{\hat y\otimes \hat F(M)}{y\in\cat{F}^c}$. Hence $\MF$ is generated as a localizing (non-ideal) subcategory by the same objects. By Condition~\eqref{it:res3}, we can replace the collection of $y\in\cat{F}^c$ by $F(X)$ for $X\in\cat{T}$. So, we have shown that $\MF$ is generated as a localizing subcategory by~$\SET{\hat F(\hat X\otimes M)}{X\in\cat{T}}$. In particular, $\hat \unit_{\cat F}$ belongs to that localizing subcategory. Applying~$\hat U$, we see that $\hat\bbE=\hat U(\hat\unit_{\cat{F}})$ belongs to the localizing subcategory of~$\cat{A}$ generated by the objects $\hat U\hat F(\hat X\otimes M)\simeq \hat X\otimes \hat{U}\hat{F}(M)\simeq \hat X\otimes \hat{\bbE}\otimes M$, for $X\in\cat{T}$. But all these objects $\hat X\otimes \hat{\bbE}\otimes M$ belong to the $\otimes$-ideal~$\overrightarrow{\cat{B}'}$ since $M$ does. In short, we have shown that $\hat\bbE=\hat U(\hat\unit_{\cat{F}})$ belongs to~$\overrightarrow{\cat{B}'}$. On the other hand, if we consider the exact sequence in~$\cat{A}$
\[
0 \to L \to \hat\unit_{\cat{T}} \otoo{\hat \eta} \hat U \hat F(\hat \unit_{\cat T})=\hat\bbE,
\]
this kernel $L=\Ker(\eta)$ belongs to~$\overrightarrow{\cat{B}}=\Ker(\hat F)$. This is always true since $F(\eta)$ is a monomorphism by the unit-counit relation. Hence $L\in\overrightarrow{\cat{B}}\subseteq\overrightarrow{\cat{B}'}$. Consequently, $\overrightarrow{\cat{B}'}$ also contains the middle term of the above sequence. This relation $\hat \unit_{\cat{T}}\in\cat{B}'$ contradicts~$\cat{B}'$ avoiding the non-zero elements of~$\cat{T}^c$.
\end{proof}

\begin{Rem}
The image $\hat{\bbE}$ of the ring-object $\bbE=U(\unit_{\cat{F}})$ in~$\cat{A}$ is $\Ker(\hat{F})$-local. Indeed, $\hat{\bbE}\cong \hat{U}(\hat\unit_{\cat{F}})\cong R\bar{U}(\hat\unit_{\cat{F}})$ belongs to the image of~$R\colon\bat{A}\to \cat{A}$.
\end{Rem}

In summary, we have proved that if $\cat{T}$ has a tt-residue field $\cat{F}$ as in Hypothesis~\ref{hyp:res} then it gives rise to a $\cat{T}^c$-maximal Serre $\otimes$-ideal $\cat{B}$ in $\mT$. Moreover, the abelian shadow $\MF$ of the tt-field $\cat{F}$ can be reconstructed from the abelian residue field $\bat{A}$ that we have constructed. Thus if there is an honest tt-residue field it gives rise to the structure we have been considering in $\MT$.

We conclude with an example.

\begin{Exa}
Let $R = (R,\gm,k)$ be a discrete valuation ring. Then $\Der(R)$ is local and $\pi^*\colon \Der(R) \to \Der(k)$ verifies Hypothesis~\ref{hyp:res}. Moreover, we know that $\Der_\gm(R)$ is minimal. Thus we can apply Corollary~\ref{cor:unique} to deduce that there is a unique $\Dperf(R)$-maximal Serre $\otimes$-ideal $\cat{B}$ of $\mmod\Dperf(R)$. By Proposition~\ref{prop:B-max} it is none other than the kernel of the induced functor
\begin{displaymath}
\widehat{\pi^*}\colon \mmod\Dperf(R) \to \mmod\Dperf(k).
\end{displaymath}
The subcategory $\cat{B}$ can be described a bit more explicitly as follows. Let $t\in\gm$ be a uniformizer and consider the corresponding triangle
\begin{displaymath}
\xymatrix{
R \ar[r]^-t & R \ar[r]^-\pi & k \ar[r] & \Sigma R
}
\end{displaymath}
giving rise to $k$. The object
\begin{displaymath}
I = \im(\hat t\colon \hat R \to \hat R)
\end{displaymath}
in $\mmod\Dperf(R)$ is none other than the $I$ from Proposition~\ref{prop:I-gens} and thus one can describe $\cat{B}$ as the Serre $\otimes$-ideal generated by $I$. The corresponding $E$ is just the residue field $k$, which is endofinite and thus pure-injective as required.
\end{Exa}

%------------------------------------------------------------------------------
\goodbreak
\begin{appendix}
\section{Generalities on Grothendieck categories}
\label{app:Grothendieck}%
\medbreak
%------------------------------------------------------------------------------

We record some general facts and constructions concerning Grothendieck categories which we used throughout. Everything in this section is standard, although there do not necessarily exist convenient references. As a result we indicate some of the proofs.  We tacitly assume that all functors are additive unless explicitly mentioned otherwise. We begin with a well-known fact.

\begin{Thm}\label{thm:GAFT}
Let $F\colon \cat{A} \to \cat{C}$ be a functor between Grothendieck categories. Then
\begin{enumerate}[\rm(a)]
\item
$F$ has a right adjoint if and only if it preserves colimits;
\item
$F$ has a left adjoint if and only if it preserves limits.
\end{enumerate}
\end{Thm}

\begin{Prop}\label{prop:closed}
Let $(\cat{A}, \otimes, \mathbf{1})$ be a symmetric monoidal Grothendieck category. If the monoidal product $\otimes$ is colimit-preserving in both variables then $\cat{A}$ is closed i.e.\ $\otimes$ admits a right adjoint in two variables
\begin{displaymath}
\ihom\colon \cat{A}^\mathrm{op} \times \cat{A} \to \cat{A}.
\end{displaymath}
\end{Prop}
\begin{proof}
For each $a\in \cat{A}$ the existence of $\ihom(a,-)$ follows from Theorem~\ref{thm:GAFT} applied to the colimit-preserving functor $a\otimes -$. Given a morphism $f\in \cat{A}(a,a')$ there is a corresponding natural transformation $a\otimes (-) \to a'\otimes (-)$, which induces a morphism
\begin{displaymath}
\ihom(a', -) \to \ihom(a, -)
\end{displaymath}
between right adjoints. It is routine to verify that these transformations assemble to give a bifunctor $\ihom$ which is right adjoint to $\otimes$ in both variables.
\end{proof}

\begin{Rem}
\label{rem:loc-coh}%
The Grothendieck categories we use in the text are \emph{locally coherent}, meaning the finitely presented objects form an \emph{abelian} subcategory~$\cat{A}\fp\subset \cat{A}$ and every object of~$\cat{A}$ is a filtered colimit of finitely presented ones (\ie $\cat{A}$ is \emph{locally finitely presented}).
\end{Rem}

We fix a locally coherent Grothendieck category~$\cat{A}$ with colimit-preserving tensor~$\otimes$ and enough flats with respect to~$\otimes$. We now want to justify the various claims made in Proposition~\ref{prop:A/B}, most of which can be found in~\cite{Grothendieck57,Gabriel62}.

Given $\cat{B}\subseteq \cat{A}\fp$ a Serre $\otimes$-ideal, we can consider
\begin{equation}
\label{eq:B^to}%
\cat{L}=\overrightarrow{\cat{B}}:=\Big\{M\in\cat{A} \,\Big|\,
\ourfrac{\displaystyle\textrm{every }f\colon P\to M\textrm{ with }P\in\cat{A}\fp}
{\displaystyle\textrm{ factors as }P\to N\to M\textrm{ with }N\in\cat{B}}
\Big\}.
\end{equation}
(Assuming $P$ finitely generated \emph{projective} does not change the
above definition in the case of~$\cat{A}=\MT$, \ie we can test the
condition with $P=\hat x$ for $x\in\cat{T}^c$.)  These are the objects
$M$ such that every finitely generated subobject of $M$ is a quotient
of an object of~$\cat{B}$, \ie the filtered colimits (in~$\cat{A}$) of
objects of~$\cat{B}$.  The above category $\overrightarrow{\cat{B}}$ is a
\emph{localizing} (\ie Serre and closed under coproducts)
$\otimes$-ideal of~$\cat{A}$. For the $\otimes$-ideal property, note
that $\overrightarrow{\cat{A}\fp}=\cat{A}$ and that $\cat{B}$ is
$\otimes$-ideal in~$\cat{A}\fp$. Moreover, the subcategory of finitely
presented objects of~$\overrightarrow{\cat{B}}$ exactly coincides with
  $\cat{B}=(\overrightarrow{\cat{B}})\fp=\overrightarrow{\cat{B}}\cap
  \cat{A}\fp$.

The Grothendieck-Gabriel quotient $\bat{A}=\cat{A}/\cat{L}$ by the Serre subcategory~$\cat{L}$ is the localization with respect to all morphisms $s\colon M\to M'$ whose kernel and cokernel belong to~$\cat{L}$. We have a localization functor~$Q$
\[
\xymatrix{
\cat{A} \ar@<-.3em>[d]_-{Q}
\\
\bat{A} \ar@<-.3em>[u]_-{R}
}
\]
which admits a right adjoint~$R$ when $\cat{L}$ is localizing. To check that the tensor descends to the quotient $\otimes\colon\bat{A}\times\bat{A}\to \bat{A}$ in such a way that $Q\colon\cat{A}\to\bat{A}$ is monoidal, it suffices to check that the collection of morphisms $s\colon M\to M'$ with kernel and cokernel in~$\cat{L}$ are preserved by tensoring with any $N\in\cat{A}$. For this, we use of course that~$\cat{L}$ is a $\otimes$-ideal. Decomposing~$s$, we can treat separately the case where $s$ is a monomorphism with cokernel in~$\cat{L}$, and the case where $s$ is an epimorphism with kernel in~$\cat{L}$; the latter is easy since $\Ker(s\otimes N)$ is a quotient of $\Ker(s)\otimes N$. For the former, it is enough to convince oneself that for any $L\in \cat{L}$ the object $\Tor_1(N,L)$ lies in~$\cat{L}$. This follows by computing with a flat resolution of $N$ and using that $\cat{L}$ is both Serre and a~$\otimes$-ideal. Now, for every flat object $M\in\cat{A}$, we have a commutative diagram
\[
\xymatrix@C=5em{
\cat{A} \ar[r]^-{M\otimes -} \ar@{->>}[d]_-{Q}
& \cat{A} \ar@{->>}[d]^-{Q}
\\
\bat{A} \ar[r]^-{Q(M)\otimes -}
& \bat{A}
}\]
whose top-right composition is exact and therefore so is the bottom functor, since $Q$ is universal among \emph{exact} functors. Therefore $Q(M)$ remains flat.

The induced monoidal structure on $\bat{A}$ is compatible with colimits, i.e.\ for every $Y\in \bat{A}$ the functor $Y\otimes (-)$ preserves colimits. Indeed, suppose $J$ is a small category and $F\colon J\to \bat{A}$ is a functor. We can assume $Y$ is of the form $QM$, for instance by using the natural isomorphism $\Id_{\bat{A}} \cong QR$. Colimit preservation is a consequence of the following string of natural isomorphisms
\begin{align*}
\colim QM \otimes F & \cong \colim QM\otimes QRF \cong \colim Q(M\otimes RF) \cong Q\colim (M\otimes RF) \\
& \cong Q(M\otimes \colim RF) \cong QM \otimes \colim QRF \cong QM\otimes \colim F.
\qedhere\end{align*}
Hence by Proposition~\ref{prop:closed}, the symmetric monoidal Grothendieck category $\bat{A}$ is closed. We denote the internal hom on $\cat{A}$ and~$\bat{A}$ by $\ihom_\cat{A}$ and $\ihom_{\bat{A}}$ respectively.

\begin{Lem}\label{lem:ihomcompatible}
There is a canonical natural isomorphism of bifunctors
\begin{displaymath}
\ihom_\cat{A}(-, R-) \cong R(\ihom_{\bat{A}}(Q-, -)).
\end{displaymath}
In particular, the functor $R$ is closed i.e.\ it preserves the internal hom.
\end{Lem}

\begin{proof}
For $M\in \cat{A}$ and $Y\in \bat{A}$ the isomorphism $\ihom_\cat{A}(M, R Y)\simeq R(\ihom_{\bat{A}}(Q M,Y))$ is given by Yoneda's Lemma from the natural isomorphism
\begin{align*}
\cat{A}(?, \ihom_\cat{A}(M, R Y)) &\cong \cat{A}(? \otimes M, R Y) \cong \bat{A}(Q(?) \otimes Q M, Y) \\
&\cong \bat{A}(Q(?), \ihom_{\bat{A}}(Q M, Y)) \cong \cat{A}(?, R(\ihom_{\bat{A}}(Q M, Y))).
\end{align*}
The final statement of the lemma then comes down to noting that for $Y,Y'\in \bat{A}$
\begin{displaymath}
R(\ihom_{\bat{A}}(Y, Y')) \cong R(\ihom_{\bat{A}}(QR Y, Y')) \cong \ihom_\cat{A}(R Y, R Y').
\qedhere
\end{displaymath}
\end{proof}

\begin{Prop}\label{prop:fp-kernel}
Let $\cat{A}$ and $\cat{C}$ be locally coherent Grothendieck categories (Remark~\ref{rem:loc-coh}) and $F\colon \cat{A} \to \cat{C}$ an exact colimit-preserving functor which sends finitely presented objects to finitely presented objects. Then $\Ker F$ is generated by finitely presented objects of $\cat{A}$, i.e.\ $\Ker F$ is the filtered colimit closure of $\Ker F \cap \cat{A}\fp$.
\end{Prop}
\begin{proof}
Replacing $\cat{A}$ by $\cat{A}$ modulo the localising subcategory generated by $\Ker F\cap \cat{A}\fp$ and $F$ by the induced functor to $\cat{C}$ we can reduce to showing that if $\Ker F \cap \cat{A}\fp$ is trivial then $F$ has no kernel. Suppose then that this is the case. We first note that $F$ is faithful on $\cat{A}\fp$. This is true for any exact conservative functor on an abelian category, which one sees by testing vanishing via the image object. As the functor $F$ preserves colimits, by Theorem~\ref{thm:GAFT}, it has a right adjoint $G$ which must preserve filtered colimits because $F$ preserves finitely presented objects. As shown above $F$ is faithful when restricted to $\cat{A}\fp$. Thus the components of the unit $\eta$ of this adjunction at finitely presented objects are monomorphisms. Since $\cat{A}$ is locally coherent it is, in particular, locally finitely presented -- every object $X$ of $\cat{A}$ is a filtered colimit of finitely presented objects. As both $F$ and $G$ preserve filtered colimits we thus see that $\eta_X$ can be written as a filtered colimit of components of $\eta$ at finitely presented objects. So we see $\eta_X$ is a filtered colimit of monomorphisms and hence, since $\cat{A}$ is Grothendieck, is itself a monomorphism. This proves $F$ is faithful and completes the argument.
\end{proof}

We also need the following easy consequence of Beck's monadicity theorem.

\begin{Prop}\label{prop:monadic}
Let $G\colon \cat{C} \to \cat{A}$ be an exact limit-preserving functor between Grothendieck categories. If $G$ is faithful then $G$ is monadic.
\end{Prop}
\begin{proof}
By Theorem~\ref{thm:GAFT} we know $G$ has a left adjoint $F$. In addition, since $\cat{C}$ is abelian and $G$ is faithful and exact, it is immediate that $G$ is conservative, $\cat{C}$ has all coequalisers (not just $G$-split ones), and $G$ preserves these coequalisers. Thus we may apply Beck's monadicity theorem~\cite{MacLane98} to deduce that $G$ is monadic.
\end{proof}

%------------------------------------------------------------------------------
\end{appendix}
%------------------------------------------------------------------------------

\bibliographystyle{alpha}%
\bibliography{BKS-biblio}

%------------------------------------------------------------------------------
%------------------------------------------------------------------------------
\end{document}